\documentclass[12pt]{amsart}
\usepackage{amssymb, latexsym, euscript}
\usepackage{amssymb}
\usepackage{latexsym}
\usepackage{amsmath}

\def\sA{{\mathfrak A}}      
\def\sD{{\mathfrak D}}      
   \def\sH{{\mathfrak H}}   
      \def\sL{{\mathfrak L}}
\def\sM{{\mathfrak M}}   \def\sN{{\mathfrak N}}   
      
\def\sS{{\mathfrak S}}

      \def\dC{{\mathbb C}}

   \def\dN{{\mathbb N}}   
      \def\dR{{\mathbb R}}
   \def\dT{{\mathbb T}}

\def\cA{{\mathcal A}}      
   \def\cE{{\mathcal E}}   
   \def\cH{{\mathcal H}}   
   \def\cK{{\mathcal K}}   \def\cL{{\mathcal L}}
   \def\cN{{\mathcal N}}   
      \def\cR{{\mathcal R}}
\def\cS{{\mathcal S}}   \def\cT{{\mathcal T}}   
      \def\cX{{\mathcal X}}

   \def\bB{{\mathbf B}}   
\def\bD{{\mathbf D}}      
   \def\bH{{\mathbf H}}

\def\dim{{\rm dim\,}}
\def\ran{{\rm ran\,}}
\def\cran{{\rm \overline{ran}\,}}
\def\dom{{\rm dom\,}}

\def\cdom{{\rm \overline{dom}\,}}

\def\cspan{{\rm \overline{span}\, }}

\def\uphar{{\upharpoonright\,}}

\def\f{\varphi}

\def\half{{\frac{1}{2}}}

\newtheorem{theorem}{Theorem}[section]
\newtheorem{lemma}[theorem]{Lemma}
\newtheorem{proposition}[theorem]{Proposition}
\newtheorem{corollary}[theorem]{Corollary}
\newtheorem{definition}[theorem]{Definition}
\newtheorem{remark}[theorem]{Remark}

\numberwithin{equation}{section}

\def\RE{{\rm Re\,}}
\def\IM{{\rm Im\,}}

\def\wt{\widetilde}
\def\wh{\widehat}

\usepackage{color}

\begin{document}
\title{Squares of symmetric operators}
\author[Yury Arlinski\u{\i}]{Yu.M. Arlinski\u{\i}}
\address {Volodymyr Dahl East Ukrainian National University, Kyiv,
 Ukraine}

\email{yury.arlinskii@gmail.com}


\begin{abstract}

Using the approach proposed in \cite{ArlCAOT2023},  
in an infinite-dimensional separable complex Hilbert space we give  abstract constructions  of  families $\{{\mathcal T}_z\}_{{\rm Im\,} z>0}$ of closed densely defined symmetric operators with the properties: (I) the domain of ${\mathcal T}_z^2$ is a core of ${\mathcal T}_z$, (II) the domain of ${\mathcal T}_z^2$ is dense but note a core of ${\mathcal T}_z$, (III) the domain of ${\mathcal T}_z^2$ is nontrivial but non-dense. For this purpose a class of maximal dissipative operators is defined and studied.
The case ${\rm dom\,} {\mathcal T}_z^2=\{0\}$ has been considered in \cite{ArlCAOT2023}.

Given a densely defined closed symmetric operator $S$, in terms of the intersection of the domain of $S$ with $\ran (S-\lambda I)$ and the projection of the domain of the adjoint $S^*$ on ${\rm ran\,} (S-\lambda I)$, $\lambda\in{\mathbb C}\setminus{\mathbb R}$, necessary and sufficient conditions for the cases (I)--(III) related to the domain of $S^2$, are obtained.

\end{abstract}
\maketitle
\thispagestyle{empty}
\section{Introduction}
Let $S$ be a densely defined closed operator in the Hilbert space $\cH$. Then the following cases  can hypothetically 
occur with the square $S^2$:
\begin{enumerate}
\item  [\rm (I)] the domain of $S^2$ is a core of $S$;
\item [\rm (II)] the domain of $S^2$ is dense in $\cH$ but note a core of $S$;
\item  [\rm (III)] the domain of $ S^2$ is nontrivial but non-dense in $\cH$;
\item  [\rm (IV)] the domain of $S^2$ is trivial;
\item [\rm (V)] $S^2f=0$ for all $f\in\dom S^2$.
\end{enumerate}

The following assertions are useful and give some criteria for the cases (I) and (IV).
\begin{lemma} \label{dbltnm}
(1) Let $S$ be a linear operator. Then for any complex number $\lambda$ the relation 
\[
\dom S\cap\ran (S-\lambda I)=(S-\lambda I)\dom S^2
\]
holds and the following statements are equivalent:
\begin{enumerate}
\def\labelenumi{\rm (\roman{enumi})}
\item $\dom S^2=\{0\}$,
\item  $\dom S\cap\ran (S-\lambda I)=\{0\}$.

\end{enumerate}

(2) Let the operator $S$ be closed and let $\lambda$ be a point of regular type of $S$ \footnote{A complex number $\lambda$ is called a point of regular type for a linear operator $S$ in a Banach space $\cX$ if there exists a positive number $c$ such that $||(S-\lambda I)f||_\cX\ge c||f||_\cX$ for all $f\in\dom S$ \cite{AG}.}. Then
the following statements are equivalent:
\begin{enumerate}
\def\labelenumi{\rm (\roman{enumi})}
\item $\dom S^2$ is a core of the operator $S$,
\item the linear manifold $\dom S\cap\ran (S-\lambda I)$ is dense in $\ran(S-\lambda I)$.

\end{enumerate}
\end{lemma}
If $\lambda$ is a point of regular type of a closed operator $S$, then $\ran (S-\lambda I)$ is a subspace (closed linear manifold) and by the bounded inverse theorem the operator $S-\lambda I$ isomorphically maps $\dom S$, equipped by the graph norm, onto $\ran (S-\lambda I)$.  It follows that if the subspace $(\ran (S-\lambda I))^\perp$ is finite-dimensional, then  the intersection $\dom S\cap\ran (S-\lambda I)$ is dense in $\ran (S-\lambda I)$ (see e.g. \cite[Chapter 1, Lemma 3] {Glazman}) and therefore $\dom S^2$ is a core of the operator $S.$

In this paper we are interested in squares of closed densely defined symmetric operators.
The case (V) for such operators holds when $\dom S^2=\ker S$. Concerning the case (IV), it is well known that
in an infinite-dimensional separable complex Hilbert space there are exist closed densely defined symmetric
operators $S$ with $\dom S^2=\{0\}$.
The first example belongs to Na\u{\i}mark \cite{Naimark1, Naimark2}.
An example of a closed densely defined semi-bounded symmetric operator $S$ in the Hilbert space $L^2(\dT)$ ($\dT$ is the unit circle) with $\dom S^2=\{0\}$ was constructed by Chernoff in \cite{Chern}. 
It is proved by Schm\"{u}dgen in \cite[Theorem~5.2]{schmud} that
for every unbounded selfadjoint operator $A$ in  infinite dimensional separable complex Hilbert space $\cH$ there exist
closed densely defined symmetric restrictions $A_1$ and $A_2$ of $A$ such that
\[
\dom A_1\cap\dom A_2=\{0\} \quad \mbox{and} \quad \dom A^2_1=\dom A^2_2=\{0\}.
\]
Brasche and Neidhardt \cite{BraNeidh} showed that the same result remains true for an arbitrary closed symmetric densely defined but non-selfadjoint operator $A$. Further results related to the existence of  closed densely defined symmetric restrictions $\wt A$ of a selfadjoint operator $A$ with the property $\dom (A\wt A)=\{0\}$ (this yields, in particular, $\dom \wt A^2=\{0\}$)
can be found in \cite{ArlKov2013} and \cite{Arl_ZAg_IEOT_2015} (see also \cite[Introduction]{ArlCAOT2023}).

Given a selfadjoint operator $A$, in \cite{schmud} have been established results related, in particular, to the existence of a closed densely defined symmetric restriction $S$ of $A$ such that $\dom S^2$ is dense and is a core/is not a core for $S$ (\cite[Theorem 4.5]{schmud}).

In our recent paper \cite{ArlCAOT2023} a new construction of families of densely defined closed symmetric operators with trivial domains of their squares has been proposed.
\textit{In the current article we use the approach of \cite{ArlCAOT2023} for constructions of families of densely defined closed symmetric operators whose squares have the properties described in items (I), (II), (III) above}. In this way maximal dissipative operators of special type play a key role.

Recall that a linear operator $T$ in a Hilbert space $\cH$ is called \textit{dissipative} if $\IM (Tf,f)\ge 0$ $\forall f\in\dom T$ and \textit{maximal dissipative} if it is dissipative and has no dissipative extensions without exit from $\cH$. It is well known that  (see e.g. \cite{Ka,Kuzhel, Straus1968})
\begin{itemize}
\item $T$ is maximal dissipative $\Longleftrightarrow$ $-T^*$ is maximal dissipative ($T^*$ is maximal accumulative);
\item the set of all regular points of maximal dissipative operator contains the open lower half-plane and 
if $T$ is maximal dissipative operator, then the resolvent admits the estimate
$
||(T-\lambda I)^{-1}||\le(|\IM \lambda|)^{-1}\;\; \forall\lambda\in\dC_-$ ($\IM \lambda<0$).
\end{itemize}
In this paper we continue the study, begun in \cite{ArlCAOT2023}, of the class of unbounded
maximal dissipative operators $T$ whose corresponding nonnegative quadratic forms
\begin{equation}\label{cbyuah1}
\gamma_T[f]:=\IM (Tf,f),\; \dom\gamma=\dom T
\end{equation}
are \textit{singular}.
 The singularity means (see \cite{Kosh, KoshDud}) that
\[
\forall f\in\dom T\;\;\exists \{f_n\}\subset\dom T: \lim\limits_{n\to\infty}f_n=f\quad\mbox{and}\quad  \lim\limits_{n\to\infty}\IM\left(Tf_n,f_n\right)=0.
\]
The class of such maximal dissipative operators we denote by $ \bD_{\rm{sing}}$.
In particular, if the linear manifold (the \textit{Hermitian domain} of $T$ \cite{Kuzhel})
\begin{equation}\label{kerna}
\sS_T:=\ker\gamma_T=\{f\in\dom T: \IM (Tf,f)=0\}
\end{equation}
is dense in $\cH$, then $T\in\bD_{\rm{sing}}$.
By \cite[Proposition 3.2]{ArlCAOT2023} the relation
\[
\sS_T=\sS_{T^*}=\{f\in\dom T\cap\dom T^*: Tf=T^*f\}=\ker(\IM T)
\]
holds.
Moreover, if $T\in\bD_{\rm{sing}}$, then $\sS_T=\dom T\cap\dom T^*$, see Theorem \ref{xfcnyck} and \cite[Lemma 3.3]{ArlCAOT2023}.
The symmetric operator $S:=T\uphar\sS_T$  is called the \textit{Hermitian part} of $T$ \cite{Kuzhel}.
If $\sS_T\ne\{0\}$ and is non-dense, then in the case when $S$ has least one finite deficiency index, we show that $T\notin  \bD_{\rm{sing}}.$

We give various equivalent conditions for a maximal dissipative operator to belong to the class $\bD_{\rm{sing}}$, see Proposition \ref{zghbl}.
For any closed symmetric operator $S$ with non-dense domain and with infinite deficiency indices we construct in Section \ref{nov3b} maximal dissipative operators of the class $\bD_{\rm{sing}}$ having Hermitian part $S$. Moreover, the corresponding abstract examples are constructed by means of the von Neumann result
\cite{Neumann1929} and of the Kre\u{\i}n \textit{shorted operator} \cite{Kr}, \cite{And, AT} (see Section \ref{nov3b}).

In order to construct closed densely defined symmetric operators with the properties (I) -- (IV) mentioned above we use the following scheme proposed in \cite{ArlCAOT2023} (see Section \ref{constrone}). Let $T$ be an unbounded maximal dissipative operator in the Hilbert space $\cH$ and let $\{\cE,\Gamma\}$ be its boundary pair connected with $T$ by the Green identity (see Subsection \ref{nov17a}), then define for each $z\in \dC_+$ (i.e. $\IM z>0$) the operator
\[
\wt \cT_z=\begin{bmatrix} T&0\cr 2i\sqrt{\IM z}\,\Gamma& zI_\cE\end{bmatrix}
,\;\;\dom \wt \cT_z=\dom T\oplus\cE,
\]
which is a maximal dissipative lifting of $T$ in the Hilbert space $\sH:=\cH\oplus\cE$. The Hermitian part $\cT_z$ of $\wt\cT_z$ is of the form
\[
\begin{array}{l}
\sS_{\wt\cT_z}=\dom \cT_z=\left\{\begin{bmatrix}f\cr-\cfrac{\Gamma f}{\sqrt{\IM z}}\,\end{bmatrix}: f\in\dom T\right\},\;
\cT_z\begin{bmatrix}f\cr-\cfrac{\Gamma f}{\sqrt{\IM z}}\end{bmatrix}=\begin{bmatrix}Tf\cr-\cfrac{\bar z\, \Gamma f}{\sqrt{\IM z}}\end{bmatrix},\, f\in\dom T.
\end{array}
\]
The symmetric operator $\cT_z$ is closed and when, in addition, the operator $T$ is a maximal accretive/maximal sectorial with vertex at the origin and the semi-angle $\alpha$, then for $\arg\in (0,\half\pi]$/$\arg z\in (0, \pi-\alpha)$ the operator $\cT_z$ is non-negative, see Theorem \ref{zghbl} and Proposition \ref{ahblajhv}.
Besides  for distinct $z_1, z_2\in\dC_+$ holds the equality
\[
\cT_{z_2}=\begin{bmatrix}I_\cH&0\cr 0&\sqrt{\cfrac{\IM z_2}{\IM z_1}}\,I_{\cE}\end{bmatrix}\left(\cT_{z_1}+\begin{bmatrix}0&0\cr 0&\cfrac{\IM(\bar z_2 z_1)}{\IM z_2}\,I_{\cE}\end{bmatrix}\right)\begin{bmatrix}I_\cH&0\cr 0&\sqrt{\cfrac{\IM z_2}{\IM z_1}}\,I_{\cE}\end{bmatrix}.
\]
\textit{The operator  $\cT_z$ is densely defined if and only if the quadratic form $\gamma_T$, given by \eqref{cbyuah1}, is singular}, see \cite[Theorem 4.2]{ArlCAOT2023} and Theorem \ref{cnzcbv}, i.e., $T\in\bD_{\rm sing}$. Moreover, if this is a case, then we proved in \cite [Theorem 4.2]{ArlCAOT2023} \\
\[
\dom\cT_z^2=\{0\}\quad\mbox{if and only if}\quad\dom T\cap\dom T^*=\{0\}.
\]
 Since $\ran (\cT_z-\bar z I)=\cH$ and $\cH\cap\dom\cT_z=\sS_T$, we apply Lemma \ref{dbltnm}.
in order to establish 
the equivalent conditions related to the properties of $\dom \cT_z^2$. In particular, we prove in Theorem \ref{cnzcbv} and Theorem \ref{ctyn29} that

({\bf a}) $\dom \cT_z^2$ is a core of $\cT_z$ $\Longleftrightarrow$ $\dom T\cap\dom T^*$ is dense in $\cH$,

({\bf b}) $\dom \cT_z^2$ is densely defined but not a core of $\cT_z$ $\Longleftrightarrow$
$$\left\{\begin{array}{l}\dom T\cap\dom T^*\ne \{0\},\; (\dom T\cap\dom T^*)^\perp\ne \{0\}\\
(\dom T\cap\dom T^*)^\perp\cap\ran\left(\IM \left(T^*-i I\right)^{-1}\right)^\half=\{0\}\end{array}\right.,$$

({\bf c}) $\dom \cT_z^2\ne \{0\}$ and is non-densely defined $\Longleftrightarrow$
$$\left\{\begin{array}{l}\dom T\cap\dom T^*\ne \{0\},\; (\dom T\cap\dom T^*)^\perp\ne \{0\}\\
(\dom T\cap\dom T^*)^\perp\cap\ran\left(\IM \left(T^*-i I\right)^{-1}\right)^\half\ne\{0\}\end{array}\right..$$

Moreover, we construct abstract example of maximal dissipative operators (Section \ref{nov3b}) and maximal sectorial and dissipative operators (Section \ref{nov28d}) of the class $\bD_{\rm sing}$, satisfying conditions ({\bf a}), ({\bf b}), ({\bf c}).


Let $S$ be  a closed symmetric operator. Set
$$\sM_z:=\ran (S-zI),\;\sN_z:=\sM_{\bar z}^\perp.$$
In Section \ref{nov3c} we consider the \textit{Shtraus extensions} \cite{Straus1968}:
\begin{equation}\label{slamb}
\dom \wt S_z=\dom S+\sN_z,\; \wt S_z(f_S+\f_z)=Sf_S+z\f_z,\; f_S\in\dom S,\;\f_z\in\sN_z
\end{equation}
of $S$. Here $z\in\dC$ is a point of regular type for $S$.
By means of the compression of $\wt S_z$ on the subspace $\sM_{\bar z}$
we prove  that any closed symmetric operator $S$ for each $z\in\dC_+$ takes the form $\cT_z$ and the Shtraus extension $\wt S_z$ of $S$ is of the form $\wt\cT_z$.  Using results described in items ({\bf b}) and ({\bf c}), this gives the possibility to get criteria in the spirit of Lemma \ref{dbltnm} 
for the cases (II) and (III) in terms of trivial or non-trivial intersections
 $$P_{\sM_{\lambda}}S^*\cap( \sM_{\lambda}\ominus(\sM_{\lambda}\cap\dom S)) ,$$
  where $\lambda\in\dC\setminus\dR$, 
see Theorem \ref{jrn29a}. Finally, similar results are established for "clones" of $S$ that have been defined and studied in \cite{CAOT2021} and \cite{ArlCAOT2023}.

Observe that the square $S^2$ of a symmetric operator is a nonnegative operator and $(S^2f,f)=||Sf||^2$ $\forall f\in\dom S^2$.
For a densely defined square $S^2$ the representations of the Friedrichs extension $(S^2)_{\rm F}$ have been considered in \cite{ArlKov2011}, \cite{ArlKov2013}, \cite{GestSchm}, \cite{RS}.  In particular, the following are equivalent (see e.g. \cite{ArlKov2011}):
\begin{enumerate}
\def\labelenumi{\rm (\roman{enumi})}
\item $\dom S^2$ is a core of $S$;
\item $(S^2)_{\rm F}=S^*S.$
\end{enumerate}
In \cite{GestSchm} it is established that
$
(S^2)_{\rm F}=\left({S\uphar\dom S^2}\right)^*(S\uphar\dom S^2)^{**}.
$ 
Besides in \cite[Proposition 3.2]{GestSchm} a description of $\dom S^2$ and a criteria for $\dom S^2$ to be a core for $S$ are established in terms of the Cayley transform
$V=(\wt S_{i}-iI)(\wt S_{i}+i I)^{-1}$ of the Shtraus extension $\wt S_i $ of $S$. 
\vskip 0.3cm

\textit{Notations}

We use the following notations. The Banach space of all bounded operators acting between Hilbert spaces $H_1$ and $H_2$ is denoted by $\bB(\cH_1,\cH_2)$ and $\bB(H):=\bB(H,H)$.
The symbols $\dom T$, $\ran T$, $\ker T$ denote
the domain, range and kernel of a linear operator $T$, respectively,
and $\cran T$ denotes the closure of the range of $T$.
The spectrum and the resolvent set of a linear operator $T$ are denoted by $\sigma(T)$ and~$\rho(T)$, respectively.
If $\sL$ is a subspace, i.e.,  a closed linear manifold of a Hilbert space, the orthogonal
projection onto $\sL$ is denoted by $P_\sL.$ The identity operator in a Hilbert space $H$ is denoted by $I_H$ and sometimes by $I$.
By $\cL^\perp$ we denote the orthogonal complement to the linear manifold $\cL$.
The notation $T\uphar \cN$ means the restriction of a linear operator $T$ to the linear manifold $\cN\subset\dom T$.
The open upper/lower half-plane of the complex plane $\dC$ are denoted by $\dC_\pm \!:=\! \{ z\!\in\!\dC: \IM z \!\gtrless\! 0\}$ and
$\dR_+:=(0,+\infty).$

\section{Symmetric operators}


Recall that a linear operator $S$ in a Hilbert space $\cH$ is called symmetric (or Hermitian) if $\IM (Sf,f)=0$ for all $f\in\dom S$. If $\dom S$ is dense, then $S$ is symmetric iff
$S\subseteq S^*$.
 Let $S$ be a closed symmetric operator in $\cH$ and denote $\wh\rho(S)$ the set of all points of regular type of $S$. Then $\wh\rho(S)\supseteq\dC\setminus\dR$. 
$\sM_\lambda$ and $\sN_\lambda$ are subspaces for each $\lambda\in\wh\rho(S)$ and $\sN_\lambda$ is called the deficiency subspace of $S$.
 The numbers
\[
\begin{array}{l}
n_+=\dim \sN_\lambda,\; \lambda\in\dC_+,\;n_-=\dim \sN_\lambda,\;\lambda\in\dC_-
\end{array}
\]
are called the deficiency indices (defect numbers) of $S$ \cite{AG,krasno}.

If $S$ is densely defined, then due to J.~von Neumann results \cite{AG}
\begin{enumerate}
\item
the domain of the adjoint operator $S^*$  admits the direct decomposition
\[
\dom S^*=\dom S\dot+\sN_\lambda\dot+\sN_{\bar\lambda},\;\lambda \in\dC\setminus\dR;
\]
\item the operator $S$ admits selfadjoint extensions in $\cH$ if and only if the deficiency indices of $S$ are equal.
\end{enumerate}
If $\cdom S\ne \cH$, then $\dom S\cap\sN_\lambda=\{0\}$, but $\dom S, $ $\sN_\lambda$, $\sN_{\bar\lambda}$ are not linearly independent \cite{krasno, Naimark3},
the adjoint of $S$ is the linear relation (multi-valued operator). Set
\[
\cH_0:=\cdom S,\;\sL:=\cH\ominus\cH_0,\; S_0:=P_{\cH_0}S.
\]
It is proved in \cite[Lemma 2]{krasno} that
$
\sL\cap\sM_{\lambda}=\{0\}\;\;\forall \lambda\in\dC\setminus\dR.
$

The operator $S$ we consider as a closed operator acting from $\cH_0$ into $\cH$. Since $\dom S$ is dense in $\cH_0$, there exists the adjoint operator $S^*$ acting from $\cH$ into $\cH_0$, $S^*$ is closed and densely defined in $\cH$.
Clearly, $S^*_0=S^*\uphar(\cH_0\cap\dom S^*)$ and (see \cite{ArlBelTsek2011})
\[
S^*\f_\lambda=P_{\cH_0}\f_\lambda,\; \f_\lambda\in\sN_\lambda.
\]
The domain $\dom S^*$ admits the decomposition $\dom S^*=\dom S+\sN_\lambda+\sN_{\bar\lambda}$ \cite[Theorem 1]{Shmulyan70}.
Moreover, (see \cite[Lemma 2, Theorem 8]{krasno})
\begin{equation}\label{xthdjy}
(S-\bar\lambda I)(S-\lambda I)^{-1}P_{\sM_\lambda}h=P_{\sM_{\bar\lambda}}h\quad\mbox{for all}\quad
h\in\sL
\end{equation}
and
\begin{equation}\label{gthtctx1}
\begin{array}{l}
\left\{\begin{array}{l}f_S+\f_\lambda+\f_{\bar\lambda}=0\\f_S\in\dom S,\;\f_\lambda\in\sN_\lambda,\;\f_{\bar\lambda}\in\sN_{\bar\lambda},\\
 \lambda \in\dC\setminus\dR
\end{array}\right.
\Longleftrightarrow
\left\{\begin{array}{l}\f_\lambda=P_{\sN_\lambda}h,\\
\f_{\bar\lambda}=-P_{\sN_{\bar\lambda}}h,\\
f_S=2i\IM\lambda(S-\lambda I)^{-1} P_{\sM_\lambda}h,\\
h\in\sL\end{array}\right..
\end{array}
\end{equation}
 Therefore
\begin{equation}\label{rhfc11}
\sN_{\bar\lambda}\cap(\dom S\dot+\sN_\lambda)=P_{\sN_{\bar\lambda}}\sL.
\end{equation}
Set
\begin{equation}\label{gjlghjcn}
\sL_\lambda:=P_{\sN_\lambda}\sL,\; \sN_\lambda':=\sN_\lambda\ominus\sL_\lambda
\end{equation}
 Then (see \cite{ArlBelTsek2011}) the operator $V_\lambda:\sL_\lambda\mapsto\sL_{\bar\lambda}$ defined by the relation
\begin{equation}\label{forbis}
V_\lambda P_{\sN_\lambda}h=P_{\sN_{\bar \lambda}} h,\; h\in\sL,\; \lambda \in\dC\setminus\dR
\end{equation}
is an isometry.
It follows from \eqref{xthdjy} and \eqref{forbis} that
\[
\begin{array}{l}
(I-V_\lambda)P_{\sN_\lambda}h=P_{\sN_\lambda}h-P_{\sN_{\bar\lambda}}h=P_{\sM_\lambda}h-P_{\sM_{\bar\lambda}}h\\
=P_{\sM_\lambda}h-(S-\bar\lambda I)(S-\lambda I)^{-1}P_{\sM_\lambda}h
=2i\IM\lambda (S-\lambda I)^{-1}P_{\sM_\lambda}h\in\dom S\;\;\forall h\in\sL.
\end{array}
\]
The subspace $\sN_\lambda'$ lies in $\cH_0$ and  is called the \textit{semi-deficiency subspace} of $S$. It is the deficiency subspace of the symmetric operator $S_0$ which is densely defined in $\cH_0$, i.e.,  $\ker (S^*-\lambda I)\cap \cH_0=\sN_\lambda'.$
It is proved by Krasnosel'ski\u{\i} \cite[Corollary to Theorem 2]{krasno} that linear manifolds $\sL_{\lambda}$ for $\lambda\in\dC\setminus\dR$ are either all closed (are subspaces) or none of them is closed. The theorem below was established in \cite[Theorem 3]{Shmulyan70} (see also \cite[Theorem 2.4.1]{ArlBelTsek2011}).
\begin{theorem}\label{ivekmzy}
 The following are equivalent:
\begin{enumerate}
\def\labelenumi{\rm (\roman{enumi})}
\item the operator $S_0$ is closed;
\item the linear manifold $\sL_\lambda$ is closed (i.e., $\sL_\lambda$ is a subspace) for at least one (then for all) $\lambda \in\dC\setminus\dR.$
\end{enumerate}
\end{theorem}
Clearly, if $\dim\sL<\infty$, then $\sL_\lambda$ is a subspace and hence the operator $S_0$ is closed.

In the next proposition we define specific non-densely defined closed symmetric operators, which we will use in Section \ref{nov12a}.
\begin{proposition}\label{ghbv1}
Let $S_0$ be an unbounded closed densely defined symmetric operator in $\cH_0$, let $\sL$ be a Hilbert space and let $L\in\bB(\cH_0, \sL)$.
\begin{enumerate}
\item Assume, in addition,
\begin{equation}\label{zwete1}
\ker L^*=\{0\},\; \ran L^*\cap \dom S^*_0=\{0\}.
\end{equation}
Then the operator
\begin{equation}\label{erst11}
S:=S_0+LS_0,\;\dom S=\dom S_0.
\end{equation}
is closed symmetric in the Hilbert space $\cH:=\cH_0\oplus\sL$ and
\begin{equation}\label{ers12}
\dom S^*\cap\sL=\{0\},
\end{equation}
where $S^*:\cH\to\cH_0$ is the adjoint to the operator $S:\cH_0\to\cH.$

\item If $S_0$ is a selfadjoint operator in $\cH_0$, then the operator
\begin{equation}\label{erst22}
S:=S_0+L, \;\dom S=\dom S_0
\end{equation}
 is closed symmetric in the Hilbert space $\cH:=\cH_0\oplus\sL$ and
\begin{equation}\label{ers33}
\dom S^*=\dom S_0\oplus\sL,\; S^*(f+g)=S_0f+L^*g,\; f\in\dom S_0,\;g\in\sL,
\end{equation}
where $S^*:\cH\to\cH_0$ is the adjoint to the operator $S:\cH_0\to\cH.$
\end{enumerate}
\end{proposition}
\begin{proof}
 (1) Because $L$ is bounded and $S_0$ is closed, the operator $S$ given by \eqref{erst11} is closed as well.
Let $g\in \sL\cap\dom S^*$. Then for all $f\in\dom S$ we have for some $h\in\cH_0$
\[
(f,h)=(Sf,g)=(S_0f+LS_0f,g)=(S_0f,g)+(S_0f,L^*g)=(S_0f, L^*g).
\]
It follows that $L^*g\in\dom S^*_0$. But \eqref{zwete1} gives $g=0$. So, \eqref{ers12} holds.

 (2) Since $L$ is bounded, the operator defined in \eqref{erst22} is closed. Hence, because $S_0$ is selfadjoint in $\cH_0$, we arrive at \eqref{ers33}.
\end{proof}

\vskip 0.3 cm

\centerline{Simple symmetric operators}
\begin{definition}\label{simplesym}
A closed symmetric operator $S$ is called simple if there is no a non-trivial reducing subspace of $S$ on which $S$ is a selfadjoint.
\end{definition}
Any closed symmetric operator $S$ admits the orthogonal decomposition
\begin{equation}\label{hfpkj}
S=S^{(0)}P_{\cH^{(0)}}\oplus S^{(1)}P_{\cH^{(1)}},
\end{equation}
where $S^{(0)}$ is a simple symmetric operator in ${\cH^{(0)}}$ and $S^{(1)}$ is a selfadjoint operator in ${\cH^{(1)}}$, $\cH={\cH^{(0)}}\oplus{\cH^{(1)}}$,
\[
{\cH^{(0)}}=\cspan\{\sN_\lambda:\lambda\in\dC\setminus\dR\},\; {\cH^{(1)}}=\bigcap\limits_{\lambda\in\dC\setminus\dR}\sM_\lambda.
\]
This statement is established by Kre\u{\i}n in \cite{Krein1949}, see also \cite[Proposition 1.1]{LT}, \cite[Lemma 6.6.4]{ArlBelTsek2011}. 
It follows that
a closed densely defined symmetric operator $S$ in $\cH$ is simple if and only if
\[
\cspan\{\sN_\lambda,\;\IM \lambda\ne 0\}=\cH.
\]

The decomposition \eqref{hfpkj} implies that
$
\sM_\lambda=\sM_\lambda(S^{(0)})\oplus\cH^{(1)},\; \lambda\in \wh\rho(S).
$


\section{Maximal dissipative operators and their Cayley transforms}\label{disoperca}
\subsection{The Cayley transform}

If $T$ is a maximal dissipative operator in $\cH$, then its  Cayley transform
\begin{equation}\label{rtkbnh}
Y_\lambda=(T-\lambda I)(T-\bar \lambda I )^{-1}\Longleftrightarrow \left\{\begin{array}{l} \psi=(T-\bar \lambda I )f\\
Y_\lambda \psi=(T-\lambda I)f\end{array}\right. f\in\dom T,\; \lambda\in\dC_+
\end{equation}
is a contraction \cite{AG}. It follows from \eqref{rtkbnh} that
\begin{equation}\label{thecayl}
I-Y_\lambda =2i\IM \lambda (T-\bar\lambda I)^{-1},\;\ran (I-Y_\lambda)=\dom T,
\end{equation}
\begin{equation}\label{theccayl}
||\psi||^2-||Y_\lambda \psi||^2=4\IM \lambda\,\IM (Tf,f),\; \psi=(T-\bar\lambda I)f,\;f\in\dom T,
\end{equation}
\begin{equation}\label{cghzo}
Y_\lambda^*=(T^*-\bar\lambda I)(T^*-\lambda I)^{-1}\Longleftrightarrow
\left\{\begin{array}{l} \phi=(T^*-\lambda I )g\\
Y_\lambda^* \phi=(T^*-\bar \lambda I)g\end{array}\right., \; g\in\dom T^*.
\end{equation}
From \eqref{rtkbnh} and \eqref{cghzo} one gets
\begin{equation}\label{jhfnyjt}
\left\{\begin{array}{l}f=(I-Y_\lambda)h\\
Tf=(\lambda I-\bar\lambda Y_\lambda)h\end{array}\right.,\;  \left\{\begin{array}{l}g=(I-Y_\lambda^*)\psi\\
T^*g=(\bar \lambda I-\lambda Y_\lambda^*)\psi\end{array}\right.,\; h, \psi\in\cH.
\end{equation}
Hence
\[
(T-\mu I)u=(\lambda-\mu)\left(I-\cfrac{\bar\lambda-\mu}{\lambda-\mu}\, Y_\lambda\right)\left(I-Y_\lambda\right)^{-1}u,\; u\in\dom T.
\]
Besides, from \eqref{rtkbnh}, \eqref{jhfnyjt} and using commutativity of $Y_\lambda$ and $Y_\mu$ one obtains
\begin{equation}\label{yvyl}
Y_\mu=\frac{\bar\lambda-\mu}{\bar\mu-\lambda}\left(Y_\lambda -
\frac{\lambda-\mu}{\bar\lambda-\mu}\, I\right)\left(I-\frac{\bar\lambda-\bar\mu}{\lambda-\bar\mu}\, Y_\lambda\right)^{-1},\; \lambda,\mu\in\dC_+.
\end{equation}

\subsection{The linear manifold $\cL_T$ and its representations}
In the following we will use the notations $D_Z:=(I-Z^*Z)^\half$ for a contraction $Z$ and $\sD_Z:=\cran D_Z.$

\begin{proposition}\label{yjdmt}
Let $T$ be a maximal dissipative operator, let $\lambda\in\dC_+$ and let $Y_\lambda$ be the Cayley transform of $T$.

(1) Define the  linear manifolds
\[
\cL_{*\lambda}:=\dom T^*+\ran D_{Y_\lambda},\;\;\cL_{\lambda}:=\dom T+\ran D_{Y^*_\lambda}.
\]
Then both of them
do not depend on $\lambda\in\dC_+$ and, moreover, they coincide with the linear manifold
$$\ran\left(\IM \left(T^*-i I\right)^{-1}\right)^\half.$$

(2) The following are equivalent:
\begin{enumerate}
\def\labelenumi{\rm (\roman{enumi})}
\item $\dom T^*\cap\ran D_{Y_\lambda}=\{0\}$;
\item $\dom T\cap\ran D_{Y^*_\lambda}=\{0\}$.
\end{enumerate}

(3) Let $\mu\ne \lambda$ and $\mu\in\dC_+$. Then
\[
\begin{array}{l}
\dom T^*\cap\ran D_{Y_\lambda}=\{0\}\Longrightarrow\dom T^*\cap\ran D_{Y_\mu}=\{0\},\\
 \dom T\cap\ran D_{Y^*_\lambda}=\{0\}\Longrightarrow \dom T\cap\ran D_{Y^*_\mu}=\{0\}.
\end{array}
\]

\end{proposition}
\begin{proof}
First we show that for fixed $\lambda\in\dC_+$ the equality $\cL_{*\lambda}=\cL_\lambda$ holds.

We use the following equality for bounded operators $F,G\in\bB(\cH)$ (see \cite[Theorem 2.2]{FW}):
\[
\ran F+\ran G=\ran \left(FF^*+GG^*\right)^\half.
\]
Since $\dom T=\ran (I-Y_\lambda)$ and $\dom T^*=\ran (I-Y^*_\lambda)$, we get
\[
\begin{array}{l}
\ran T^*+\ran D_{Y_\lambda}= \ran (I-Y^*_\lambda)+\ran D_{Y_\lambda}=\ran\left((I-Y^*_\lambda)(I-Y_\lambda)+D^2_{Y_\lambda}\right)^\half\\
=\ran\left((I-Y^*_\lambda)(I-Y_\lambda)+I-Y^*_\lambda Y_\lambda\right)^\half=\ran\left(2I- (Y^*_\lambda+ Y_\lambda))\right)^\half.
\end{array}
\]
Similarly
\[
\ran T+\ran D_{Y^*_\lambda}=\ran\left(2I- (Y^*_\lambda+ Y_\lambda))\right)^\half=\ran\left(I- \cfrac{Y^*_\lambda+ Y_\lambda}{2}\right)^\half.
\]
Thus,
\begin{equation}\label{htujy}
\cL_{*\lambda}=\cL_\lambda=\ran\left (I-\RE Y_\lambda\right)^{\half}.
\end{equation}
Let us show that $\cL_{*\lambda}$ does not depend on $\lambda\in\dC_+$.

Let $Y$ be a contraction on the Hilbert space $\cH$ and let $a$ and $\eta$ be complex numbers, $|a|<1$, $|\eta|=1$.
Set
\[
\wh Y=\eta(Y-aI)(I-\bar a Y)^{-1}.
\]
Then $\wh Y$ is a contraction and $\ran D_{\wh Y}=(I-aY^*)^{-1}\ran D_Y$ (see \cite[Chapter VI, (1.7), (1.8)]{SF}).
 Hence, from \eqref{yvyl}
\begin{equation}\label{dczp}
 \ran D_{Y_\mu}=(I-aY^*_\lambda)^{-1}\ran D_{Y_\lambda},\; \; a=\frac{\lambda-\mu}{\bar\lambda-\mu},\;\;\lambda,\mu\in\dC_+.
\end{equation}
From \eqref{cghzo} one gets
\[
(I-aY^*_\lambda)^{-1}=\cfrac{1}{1-a}(T^*-\lambda I)(T^*-\xi I)^{-1},\; \xi=\cfrac{\lambda-a\bar\lambda}{1-a}.
\]
Since
\[
(T^*-\lambda I)(T^*-\xi I)^{-1}=I+(\xi-\lambda)(T^*-\xi I)^{-1},
\]
and for any $f\in\cH$
\[
(I-aY^*_\lambda)^{-1}D_{Y_\lambda} f=\cfrac{1}{1-a}\left(D_{Y_\lambda}f+(\xi-\lambda)(T^*-\xi I)^{-1}D_{Y_\lambda}f\right)\in\dom T^*+\ran D_{Y_\lambda},
\]
the relation for $\ran D_{Y_\mu}$ in \eqref{dczp} implies the inclusion
\[
\ran D_{Y_\mu}\subset \dom T^*+\ran D_{Y_\lambda}\;\;\forall \lambda,\mu\in\dC_+.
\]
Thus, the latter yields that $\cL_\lambda=\cL_\mu$ and, moreover, the implication
$\dom T^*\cap\ran D_{Y_\lambda}=\{0\}\Longrightarrow\dom T^*\cap\ran D_{Y_\mu}=\{0\}$ holds. 

Let us prove the equivalence (i) $\Longleftrightarrow$ (ii) in (2). 
Note that $\ker (I-Y_\lambda)=\ker(I-Y^*_\lambda)=\{0\}.$
Suppose that
\[
 D_{Y_\lambda ^*} h=(I-Y_\lambda)g,\; h,g\in \sH.
\]
Then due to the equality $Y^*_\lambda D_{Y_\lambda ^*}=D_{Y_\lambda} Y_\lambda^*$
we get
\begin{multline*}
Y_\lambda^*D_{Y_\lambda ^*} h= D_{Y_\lambda}  Y_\lambda^*h=Y_\lambda^*(I-Y_\lambda)g\\
= D_{Y_\lambda}^2 g-(I-Y_\lambda^*)g\Longleftrightarrow   D_{Y_\lambda }\left( D_{Y_\lambda }g-Y_\lambda^* h\right)=(I-Y_\lambda^*)g.
\end{multline*}
It follows that
$$\ran  D_{Y_\lambda }\cap\ran (I-Y_\lambda^*)=\{0\}\Longrightarrow \ran D_{Y_\lambda ^*}\cap\ran (I-Y_\lambda)=\{0\}.$$
Similarly
$$\ran D_{Y_\lambda ^*}\cap\ran (I-Y_\lambda)=\{0\}\Longrightarrow \ran  D_{Y_\lambda }\cap\ran (I-Y_\lambda^*)=\{0\}.$$

A calculation gives the equality
\begin{equation}\label{rjhtymrdf}
I-\RE Y_\lambda=-2\IM \lambda\, \IM \left[(T-\bar\lambda I)^{-1}\right]=2\IM\lambda\,\IM\left[(T^*-\lambda I)^{-1}\right],\; \lambda\in\dC_+.
\end{equation}
Now from \eqref{htujy} and \eqref{rjhtymrdf} we get
\[
\cL_{*\lambda}=\cL_\lambda=\ran\left(\IM \left(-T-i I\right)^{-1}\right)^\half=\ran\left(\IM \left(T^*-i I\right)^{-1}\right)^\half.
\]
Besides from \eqref{htujy}
\[
\cL_{*\lambda}=\cL_\lambda\supset \dom T+\dom T^*.
\]
The proof is complete.
\end{proof}
In the following we will use the notation
\begin{equation}\label{ctyn29a}
\cL_T:=\ran\left(\IM \left(-T-i I\right)^{-1}\right)^\half=\ran\left(\IM \left(T^*-i I\right)^{-1}\right)^\half
\end{equation}
Due to Proposition \ref{yjdmt} the linear manifold $\cL_T$ admits the representations
\begin{equation}\label{orn11}
\cL_T=\dom T^*+\ran D_{Y_\lambda^*}=\dom T+\ran D_{Y_\lambda}\;\;\forall \lambda\in\dC_+
\end{equation}
and the inclusion $\dom T+\dom T^*\subseteq\cL_T$ holds.

\subsection{Hermitian domains and Hermitian parts of maximal dissipative operators}

\begin{definition} \label{thvbnljv}\cite{Kuzhel}. Let $T$ be a maximal dissipative operator, then the set $\sS_T$, defined in
\eqref{kerna}
is called the Hermitian domain of $T$.
\end{definition}

\begin{proposition}\label{xnjnj}
Let $T$ be a maximal dissipative operator in the Hilbert space $\cH$.  Then
\begin{enumerate}

\item the Hermitian domain $\sS_T$ is a linear manifold and
\begin{equation}\label{kerna2}
\sS_T=\sS_{T^*}=\{f\in\dom T\cap\dom T^*: Tf=T^*f\}=\ker(\IM T);
\end{equation}
hence in the case $\sS_T\ne\{0\}$, the operator $S$ given by
\begin{equation}\label{heroper}
\dom S= \sS_T,\; S=T\uphar\dom S
\end{equation}
is  symmetric and closed  and both operators $T$ and $T^*$ are extensions of $S$;
\item the following are equivalent:
\begin{enumerate}
\item $\sS_T$ is dense,
\item $\sD_{Y_{\lambda}}\cap\ran (I-Y_{\lambda}^*)=\{0\},$
\item $ \sD_{Y_{\lambda}^*}\cap\ran (I-Y_{\lambda})=\{0\}$,
where $Y_\lambda$ is the Cayley transform of $T$;
\end{enumerate}
 \item  if $\sS_T$ is dense in $\cH$, then $\sS_T=\dom T\cap\dom T^*$; 
\item for each $\lambda\in\dC_+$ holds the equalities
\begin{equation}\label{jhnljg11}
\sL=\cH\ominus\sS_T=(T^*-\lambda I)(\dom T^*\cap \sD_{Y_{\lambda}})=(T-\bar \lambda I)(\dom T\cap \sD_{Y^*_{\lambda}}).
\end{equation}
\end{enumerate}

\end{proposition}

\begin{proof}
The statement (1) is proved in \cite[Proposition 3.2]{ArlCAOT2023}. 

(2) From \eqref{theccayl} and \eqref{kerna2} it follows that
\begin{equation}\label{kerna3}
\sS_{T}=\sS_{T^*}=(I-Y_\lambda)\ker D_{Y_\lambda}=(I-Y_\lambda^*)\ker D_{Y_\lambda^*}
\end{equation}
we get that (a)$\Longleftrightarrow $ (b) $\Longleftrightarrow$ (c).

Statement (3) follows from \eqref{kerna2}.

(4) Equalities in \eqref{kerna3} yield that
\[
\sS^\perp_T=\{ f: (I-Y^*_\lambda)f\in\sD_{Y_\lambda}\}=\{f: (I-Y_\lambda)f\in\sD_{Y^*_\lambda}\}
\]
Using \eqref{thecayl} we get
\[
\sS^\perp_T=\{ f: (T^*-\lambda I)^{-1}f\in\sD_{Y_\lambda}\}=\{f: (T-\bar\lambda I)^{-1}f\in\sD_{Y^*_\lambda}\}.
\]
Thus, \eqref{jhnljg11} holds true.
\end{proof}
 The operator $S$ defined in \eqref{heroper} is called the Hermitian part of $T$ \cite{Kuzhel}.

\begin{theorem} \label{dljgec}
Let $S$ be a closed symmetric operator and let $\lambda\in\dC_+$. Then the formulas
\begin{equation}\label{ljghfci}
\left\{ \begin{array}{l}\dom T=\dom S\dot+(I-M_\lambda)\sN_\lambda\\
T(f_S+(I-M_\lambda)\f_\lambda)=Sf_S+\lambda\f_\lambda-\bar\lambda M_\lambda\f_\lambda\qquad (f_S\in\dom S,\;\f_\lambda\in\sN_\lambda)
\end{array}\right.
\end{equation}
establish a one-to-one correspondence between all maximal dissipative extensions $T$ of $S$ such that $\sS_T=\dom S$ and all
contractions $M_\lambda\in\bB(\sN_\lambda,\sN_{\bar\lambda})$ such that
\begin{equation}\label{dfyn}
||M_\lambda\f_\lambda||< ||\f_\lambda||\;\; \forall \f_\lambda\in\sN_\lambda\setminus\{0\}.
\end{equation}
The adjoint operator can be described as follows
\begin{equation}\label{ljghfci2}
\left\{ \begin{array}{l}\dom T^*=\dom S\dot+(I-M_\lambda^*)\sN_{\bar\lambda}\\
T^*(f_S+(I-M_\lambda ^*)\f_{\bar \lambda})=Sf_S+\bar\lambda\f_{\bar\lambda}-\lambda M_\lambda ^*\f_{\bar\lambda}\qquad (f_S\in\dom S,\;\f_{\bar\lambda}\in\sN_{\bar\lambda})
\end{array}\right..
\end{equation}
\end{theorem}
\begin{proof}
By \cite[Theorem 1.1]{Straus1968} the formulas \eqref{ljghfci} establish a one-to-one correspondence between all maximal dissipative extensions of $S$ and all contractions
$M_\lambda\in\bB(\sN_\lambda,\sN_{\bar\lambda})$ such that
\begin{equation}\label{admiss1}
(M_\lambda-V_\lambda)\psi_\lambda\ne 0\;\; \forall \psi_\lambda\in\sL_\lambda\setminus\{0\},
\end{equation}
where $V_\lambda$ is defined by \eqref{forbis}.
Then from \eqref{ljghfci}
\[
\IM(T(f_S+(I-M_\lambda)\f_\lambda), f_S+(I-M_\lambda)\f_\lambda)=\IM\lambda\,||D_{M_\lambda}\f_\lambda||^2,\;
 f_S\in\dom S,\;\f_\lambda\in\sN_\lambda.
\]
It follows that if $f= f_S+(I-M_\lambda)\f_\lambda$, then $\IM (Tf,f)=0$ if and only if $||\f_\lambda||=||M_\lambda\f_\lambda||$. Hence
$f\in\sS_T$ if and only if $f=f_S$ $\Longleftrightarrow$
$||M_\lambda\f_\lambda||< ||\f_\lambda||$ for all $\f_\lambda\in\sN_\lambda\setminus\{0\}$. Finally we note that, because the operator $V_\lambda$ is an isometry, condition \eqref{dfyn}
yields \eqref{admiss1}.
\end{proof}

Observe that if $T$ is a maximal dissipative extension of $S$ of the form \eqref{ljghfci} with condition \eqref{dfyn}, then the Cayley transform \eqref{rtkbnh}
is the orthogonal sum
\begin{equation}\label{jgthvl}
Y_\lambda=U_\lambda \oplus M_\lambda,\;U_\lambda=(S-\lambda I)(S-\bar\lambda I)^{-1}:\sM_{\bar\lambda}\to\sM_\lambda,\; M_\lambda:=Y_\lambda\uphar\sN_\lambda:\sN_\lambda\to\sN_{\bar\lambda}.
\end{equation}
Since $U_\lambda$ is an isometry, we get that
\begin{equation}\label{defoperl}
D^2_{Y_\lambda}=D^2_{M_\lambda}P_{\sN_\lambda},\;D^2_{Y^*_\lambda}=D^2_{M^*_\lambda}P_{\sN_{\bar\lambda}}.
\end{equation}
\begin{remark}\label{dctulf}
If $T$ is a maximal dissipative extension of a non-dense symmetric operator $S$ and $\sL=(\dom S)^\perp$, then from \eqref{ljghfci}, \eqref{ljghfci2},  \eqref{xthdjy}, \eqref{gthtctx1}, and \eqref{gjlghjcn} it follows that 
\[
\dom T\cap\sN_{\bar\lambda}= \left(P_{\sN_{\bar\lambda}}-M_\lambda P_{\sN_\lambda}\right)\sL,
\;\dom T^*\cap\sN_{\lambda}=\left(P_{\sN_{\lambda}}-M^*_\lambda P_{\sN_{\bar\lambda}}\right)\sL\;\;\forall\lambda\in\dC_+.
\]
\end{remark}
\centerline{Simple maximal dissipative operators}
\begin{definition}\label{simple}

 A maximal dissipative operator $T$ is called simple if there is no a non-trivial reducing subspace of $T$ on which $T$ is a selfadjoint.

\end{definition}
Recall that a contraction $Y$ acting in a Hilbert space $\cH$ is called \textit{completely nonunitary} (c.n.u.) \cite{SF} if there is no a reducing subspace of $Y$ in which the operator $Y$ is unitary.
Every contraction $Y$ in $\cH$ admits the orthogonal decomposition  $Y=Y^{(0)}P_{\cH^{(0)}}\oplus Y^{(1)}P_{\cH^{(1)}}$, where $Y^{(0)}$ is a c.n.u. contraction in the ${\cH^{(0)}}$ and $Y^{(1)}$ is a unitary operator in ${\cH^{(1)}}$, $\cH={\cH^{(0)}}\oplus {\cH^{(1)}},$ see \cite[Theorem I.3.2]{SF}. The operator $Y^{(0)}$ is called completely nonunitary part of $Y$.

If $T$ is a  maximal dissipative operator, then the usage of the Cayley transform leads to the to decomposition
$$T=T^{(0)}P_{\cH^{(0)}}\oplus T^{(1)}P_{\cH^{(1)}},$$
where $T^{(0)}$ is a simple maximal dissipative operator in ${\cH^{(0)}}$ and $T^{(1)}$ is a selfadjoint in ${\cH^{(1)}}$.

Note that if $Y$ is a contraction in $\cH$ and if $\cH_1$ is an invariant subspace of $Y$ and $U_1:=Y\uphar\cH_1$ is a unitary operator, then $\cH_1$ reduces $Y$.

From Definition \ref{simplesym} and equalities \eqref{kerna}, \eqref{kerna2}, and \eqref{heroper} one can easily derive the following statement.
\begin{proposition}\label{sbvgk}
A maximal dissipative operator $T$ is simple if and only if the Hermitian part $S$ of $T$ defined in \eqref{heroper} is simple.
\end{proposition}


\subsection{Boundary pairs associated with a maximal dissipative operator}\label{nov17a}
\subsubsection{Rigged Hilbert space associated with a closed operator}
Let $X$ be a closed and densely defined operator in the Hilbert space $\cH$. Denote by $\cH_X^+$ the Hilbert space $\dom X$ equipped by the inner product
\[
(f,g)_+=(Xf,Xg)+(f,g).
\]
Let $\cH_X^+\subset\cH\subset\cH_X^-$  be the corresponding rigged Hilbert space \cite{Ber}, where $\cH_X^-$ is the closure of $\cH$ with respect to the norm
\[
   \|h\|_-:=\sup\limits_{\varphi\in \cH_X^+,\|\varphi\|_{   {+}}=1}|(\varphi, h)|.
\]
 $\cH_X^-$ is the Hilbert space of all linear functionals on $\cH_X^+$ bounded with respect to $\|\cdot\|_+$ 
each of which is determined by a vector $h\in \cH_X^-$ via (see  e.g. \cite{Ber})
\[
   l_h(\varphi)=(\varphi, h),\quad \varphi\in \cH_X^+.
\]
Clearly,  $ \|f\| \!\le\! \|f\|_+$ and $\|Xf\| \!\le\! \|f\|_+$,  $f\in \cH_X^+$, hence $X\!\in\!\bB(\cH_X^+,\cH )$. Let $X^\times\!\in\!\bB(\cH, \cH_X^-)$ be the adjoint operator, ,~i.e.,
\[
   (Xf,g)=(f, X^\times g), \quad f\in \cH_X^+,\ g\in \cH.
\]
The operator $X^\times\in\bB(\cH, \cH_X^-)$ is the continuation of $X^*$ onto $\cH$. If $z\in \rho(X)$, then the resolvent
$(X-z I_\cH)^{-1}$ continuously maps $\cH$ onto $\cH_X^+$. Its adjoint is $(X^\times -\bar z I_\cH)^{-1}$:
\[
\left ((X-z I_\cH)^{-1}f,h\right)=\left(f, (X^\times -\bar z I_\cH)^{-1}h\right)\;\;\forall f\in\cH,\;\forall h\in\cH_X^-.
\]
The operator $(X^\times -\bar z I_\cH)^{-1}$ continuously maps $\cH_X^-$ onto $\cH$.
The imaginary part $\cfrac{1}{2i}(X-X^\times)$ belongs to
$\bB(\cH_X^+,\cH_X^-)$.
Observe that (see \cite{Shmulyan67}) for any bounded operator $\Gamma$, acting from $\cH_X^+$ into a Hilbert space $\cE$, for its adjoint $\Gamma^\times:\cE\to\cH_X^-$ holds
the relation
\begin{equation}\label{shmuly}
\ran \Gamma^\times=\left\{ h \in \cH_X^-:\sup\limits_{f\in \cH_X^+\setminus\ker\Gamma}\cfrac{|(f,h)|}{||\Gamma f||_\cE}<\infty \right\}.
\end{equation}
Due to  \cite[Lemma 3]{Shmulyan67}, for a bounded operator $A$ acting in a Hilbert space $\sH$, the following statement holds: the vector
 $h$ belongs to $\ran A$ if and only if
\[
\sup\limits_{f\in \sH\setminus\{0\}}\cfrac{|(f,h)|}{||A^*f||}<\infty.
\]
Hence, the equality $\ran \Gamma^\times   \cap \cH=\{0\}$ holds if and only if the operator $\Gamma$, as acting from $\cH$ into $\cE$, is \textit{strict singular}, i.e., $\dom \Gamma^*=\{0\}$ (see also \cite{Ota1987}). 
\subsubsection{The Green identity and boundary pairs}
Let $T$ be a maximal dissipative operator in the Hilbert space $\cH.$ It is proved in \cite{Straus1960} (see also \cite[Lemma 3.1]{BMNW2020}) that there exists a Hilbert space $\cE$ and a linear operator $\Gamma:\dom T\to\cE$ 
such that $\cran\Gamma =\cE$ and the Green (the Lagrange) identity
\begin{equation} \label{greens}
(Tf,g)-(f, Tg)=2i(\Gamma f, \Gamma g)_\cE\;\;\forall f,g\in\dom T
\end{equation}
holds. It follows that (see \cite[Lemma 3.3]{ArlCAOT2023})
$\Gamma\in\bB(\cH_T^+,\cE)$ and holds the equalities
\begin{equation}\label{vybvfz}
\cfrac{1}{2i}(T-T^\times)=\Gamma^\times    \Gamma,\; \ker \Gamma=\sS_T,\;\cran \Gamma^\times\cap \cH =\sS_T^\perp,
\end{equation}
where $\sS_T$ is the Hermitian domain of $T$ defined in \eqref{kerna}, $\Gamma^\times   \in\bB(\cE,\cH_T^-)$ is the adjoint to $\Gamma$, i.e., $(\Gamma f,e)_\cE=(f,\Gamma^\times e)$ for all $f\in\cH^+_T$ and for all $e\in\cE$ \cite{Ber}.

\begin{lemma}\cite[Lemma 3.3]{ArlCAOT2023}\label{ygkcvby}
If
\begin{equation}\label{yekm}
\ran \Gamma^\times\cap\cH=\{0\},
\end{equation}
then holds the equality
\[
\sS_T =\dom T\cap\dom T^*.
\]
\end{lemma}

The Hilbert space $\cE$ and the operator $\Gamma$ are called the \textit{boundary space} and the \textit{boundary operator}, respectively, associated with $T$, see, e.g. \cite{Straus1960}. We will call $\{\cE,\Gamma\}$ a \textit{boundary pair} for $T$.

Observe that from \eqref{greens} it follows the equality
\begin{equation} \label{ghtct}
\IM (Tf,f)=||\Gamma f||_{\cE}^2,\; f\in\dom T
\end{equation}
and from \eqref{ghtct} one gets that $\sS_T$ is dense in $\cH$ if and only if $\cran \Gamma^\times\cap\cH=\{0\}$;

Due to \eqref{shmuly} 
the condition \eqref{yekm} is equivalent to the condition
\[
\sup\limits_{f\in \dom T\setminus\ker\sS_T}\cfrac{|(f,h)|^2}{\IM (Tf,f)}=+\infty\;\;\forall h\in\cH.
\]
\begin{remark}\label{vjtyjdp}
(1) There exist a Hilbert space $\cE_*$ and $\Gamma_*\in\bB(\cH_{T^*}^+,\cE_*)$ such that $\cran\Gamma_*=\cE_*$ and
\[
(T^*u,v)-(u, T^*v)=-2i(\Gamma_*
u, \Gamma_* v)_{\cE_*}\;\;\forall u,v\in\dom T^*.
\]
(2) From \eqref{theccayl} it follows that one can choose
\[
\begin{array}{l}
\cE_\lambda:=\sD_{Y_\lambda},\; \Gamma_\lambda f:=\cfrac{1}{2\sqrt{\IM\lambda}}\,D_{Y_\lambda}(T-\bar\lambda I)f,\; f\in\dom T=\cH_T^+,\\
\cE_{*\lambda}=\sD_{Y^*_{\lambda}},\; \Gamma_{*\lambda}h:=-\cfrac{1}{2\sqrt{\IM\lambda}}\,D_{Y^*_\lambda}(T^*-\lambda I)h,\; h\in\dom T^*=\cH_{T^*}^+,
\end{array}
\]
where $\lambda\in\dC_+$ and $Y_\lambda $ is the Cayley transform of $T$ given by \eqref{rtkbnh}.
Moreover, if $\{\cE,\Gamma\}$ ($\{\cE_*,\Gamma_*\}$) is a boundary pair associated with $T$ (with $T^*$, respectively), then $\Gamma =U_\lambda\Gamma_\lambda,$
($\Gamma_* =U_{*\lambda}\Gamma_{*\lambda},)$
where $U_\lambda\in\bB(\cE_\lambda,\cE)$  ($U_*\in\bB(\cE_{*\lambda},\cE_*))$
is a unitary operator.

Therefore
\begin{equation}\label{ufvvfl}
\left\{\begin{array}{l}
\Gamma_\lambda^\times h=\cfrac{1}{2\sqrt{\IM\lambda}}(T^\times-\lambda I)D_{Y_\lambda}h,\; h\in\sD_{Y_\lambda},\\
\Gamma_{*\lambda}^\times f=-\cfrac{1}{2\sqrt{\IM\lambda}}((T^*)^\times-\bar\lambda I)D_{Y^*_\lambda}f,\; f\in\sD_{Y^*_\lambda}.
\end{array}\right.
\end{equation}
It follows that
$$(T^\times -\lambda I)^{-1}\ran\Gamma_\lambda^\times=\ran D_{Y_\lambda}=\ran D_{M_\lambda}\subseteq\sN_\lambda\subset\dom S^*,$$
where $S=T\uphar\sS_T$ is the Hermitian part of $T$ and $S^*:\cH\to\overline{\sS_T}$ is the adjoint to $S$.
Hence
\[
\begin{array}{l}
\ran \Gamma^\times_\lambda\cap\cH=\{0\}\Longleftrightarrow\dom T^*\cap \ran D_{Y_\lambda} =\{0\},\\
\ran \Gamma^\times_{*\lambda}\cap\cH=\{0\}\Longleftrightarrow\dom T\cap \ran D_{Y^*_\lambda}=\{0\}.
\end{array}
\]
 (3) From the definition of a boundary operator it follows that
 \[
 \IM (T^*-i I)^{-1}=\cfrac{1}{2i}\left((T^* -iI)^{-1}-(T +iI)^{-1}\right)=(T^\times -iI)^{-1}(\Gamma^\times \Gamma +I)(T+i I)^{-1}.
 \]
 Then \eqref{ctyn29a} yields the equality
 \[
\cL_T=\ran\left((T^\times -iI)^{-1}(\Gamma^\times \Gamma +I)(T+i I)^{-1}\right)^\half.
 \]
\end{remark}

\section{Maximal dissipative operators of the class $\bD_{{\rm{sing}}}$}
We introduce here a class of maximal dissipative operators, which plays a key role in our further constructions. In fact this class has been already partly studied in our paper \cite{ArlCAOT2023}.

\begin{definition} \cite[Definition 5.1]{Kosh}
A nonnegative quadratic form $\mathfrak{a}$ in a Hilbert space is called singular if for any $\f\in\dom \mathfrak{a}$ there exists a sequence $\{\f_n\}\subset\dom \mathfrak{a}$ such that
$$\lim\limits_{n\to\infty}\f_n=0\quad\mbox{and}\quad  \lim\limits_{n\to\infty}\mathfrak{a}[\f-\f_n]=0.$$
\end{definition}

The equivalent definition: a nonnegative quadratic form $\mathfrak{a}$ in a Hilbert space is singular if for any $\f\in\dom \mathfrak{a}$ there exists a sequence $\{\psi_n\}\subset\dom \mathfrak{a}$ such that
$$\lim\limits_{n\to\infty}\psi_n=\f\quad\mbox{and}\quad  \lim\limits_{n\to\infty}\mathfrak{a}[\psi_n]=0.$$

In particular, \textit{if $\ker\mathfrak{a}$ is dense, then $\mathfrak{a}$ is singular}. Actually, for any $f$ in the Hilbert space, there exists a sequence $\{\psi_n\}\subset\ker \mathfrak{a}$ such that $\lim\limits_{n\to\infty}\psi_n=f$, but $\mathfrak{a}[\psi_n]=0$ for all $n\in\dN.$

 According
 \cite[Theorem 2]{Ota1987} if $\sA$ is a linear operator between Hilbert spaces with dense domain and dense range, then the quadratic form $\mathfrak{a}[\f]:=||\sA\f||^2,$
$\dom\mathfrak{a}=\dom \sA$ is singular if and only if $\dom \sA^*=\{0\}$.

\begin{definition}\label{nov10a}
A maximal dissipative operator $T$ for which the quadratic form
\[
\gamma_T[f]:=\IM (Tf,f),\; f\in\dom\gamma_T=\dom T
\]
is singular
 will be assigned to the class $\bD_{{\rm{sing}}}$.
\end{definition}

\begin{proposition} \label{zghbl} \cite[Proposition 3.5]{ArlCAOT2023}.
Let $T$ be a maximal dissipative operator in the Hilbert space $\cH$ and let $Y_\lambda=(T-\lambda I)(T-\bar\lambda I)^{-1}$ be the Cayley transform of $T$ ($\lambda\in\dC_+$).
Let $\left\{\cE,\Gamma\right\}$ and
$\left\{\cE_*,\Gamma_*\right\}$ be boundary pairs associated with $T$ and $T^*$, respectively.
Then the following are equivalent:
\begin{enumerate}
\def\labelenumi{\rm (\roman{enumi})}
\item  $T\in\bD_{{\rm{sing}}}$;
\item $\dom\Gamma^*=\{0\}$;
\item 
$\ran \Gamma^\times\cap \cH=\{0\}$;
\item $\ran D_{Y_\lambda}\cap\dom T^*
=\{0\};$
\item $-T^*\in\bD_{{\rm{sing}}};$
\item $\dom \Gamma^*_*=\{0\}$;
\item $\ran \Gamma^\times_*\cap\cH=\{0\}$;
\item $\ran D_{Y_\lambda^*}\cap\dom T
=\{0\}.$
\end{enumerate}
\end{proposition}

From Proposition \ref{yjdmt}, Proposition \ref{zghbl} and equalities \eqref{ctyn29a} and \eqref{orn11} we arrive at the following statement.
\begin{corollary}\label{octob9a}
 The following are equivalent:
\begin{enumerate}
\def\labelenumi{\rm (\roman{enumi})}
\item $T\in \bD_{{\rm{sing}}}$;
\item at least for one (then for all) $\lambda\in\dC_+$  the linear manifold $\cL_T$ admits the direct decompositions
\[
\cL_T= \dom T^*\dot+\ran D_{Y_\lambda}=\dom T
\dot + \ran D_{Y^*_\lambda}.
\]
\end{enumerate}
\end{corollary}

\begin{theorem} \label{xfcnyck}
Let $T$ be a maximal dissipative operator in the Hilbert space $\cH$, let the linear manifold $\sS_T$ be defined by \eqref{kerna} and let
$Y_\lambda=(T-\lambda I)(T-\bar\lambda I)^{-1}$ be the Cayley transform of $T$ ($\IM\lambda>0$).

(1) Suppose that $\sS_T\ne\{0\}$ and set $\sL=\cH\ominus\sS_T.$ Then the following are equivalent:
\begin{enumerate}
\def\labelenumi{\rm (\roman{enumi})}
\item $T\in\bD_{{\rm{sing}}}$;

\item holds the equality
\[
\ran  D_{M_\lambda ^*}\cap (P_{\sN_{\bar\lambda}}-M_\lambda P_{\sN_\lambda})\sL=\{0\};
\]
\item holds the equality
\[
\ran  D_{M_\lambda }\cap (P_{\sN_{\lambda}}-M_\lambda ^* P_{\sN_{\bar\lambda}})\sL=\{0\}.
\]
where $\sN_\mu=\left((T-\bar \mu I)\sS_T\right)^\perp$ ($\IM \mu\ne 0$) is the deficiency subspace of the Hermitian part of $T$ (given by \eqref{heroper}) and $M_\lambda=Y_\lambda\uphar\sN_\lambda.$
\end{enumerate}

In particular, the following statements hold true
\begin{enumerate}
\def\labelenumi{\rm (\roman{enumi})}
\item [{\rm (a)}]
 if $\sS_T$ is dense in $\cH$, then $T\in\bD_{\rm{sing}}$;
\item [{\rm (b)}] if $\sS_T\ne\{0\}$, $\sS_T$ is non-dense in $\cH$, and $||M_\lambda||<1$ for some $\lambda$, $\IM \lambda>0$ (this is valid, for instance, if $\dim\sN_\lambda<\infty$), then $T\notin \bD_{\rm sing}$.

\end{enumerate}
(2) If $T\in\bD_{\rm{sing}}$, then holds the equality
\[
\sS_T =\dom T\cap\dom T^*.
\]
\end{theorem}

\begin{proof}
 (1) Let $S$ be given by \eqref{heroper}. Then $T$ can be described by \eqref{ljghfci}, where $M_\lambda$ is defined by \eqref{jgthvl}. We will use Proposition \ref{zghbl}.

From \eqref{defoperl} the inclusion $f\in\ran (I-Y_\lambda)\cap\ran D_{Y_\lambda ^*} $   can be rewritten as follows
\begin{equation} \label{gthtctx2}
f= D_{M_\lambda ^*}\f_{\bar\lambda}=f_S+(I-M_\lambda)\psi_\lambda, \; f_S\in\dom S,\;\f_{\bar\lambda}\in\sN_{\bar \lambda},\;\psi_\lambda\in\sN_\lambda.
\end{equation}

(a) Suppose that $\sS_T=\dom S$ is dense. Then because $\ker D_{M^*_\lambda}=\{0\}$ one gets
\[
f_S=0,\; \psi_\lambda=0 ,\;\f_{\bar \lambda}=0.
\]
It follows that
$$\ran (I-Y_\lambda)\cap\ran D_{Y_\lambda ^*}=\{0\}.$$

(b) Suppose $\sS_T\ne\{0\}$, $\sS_T$ is non-dense. Then, using \eqref{gthtctx1}, from \eqref{gthtctx2} we get
\[
\psi_\lambda=P_{\sN_\lambda}h,\;  D_{M_\lambda ^*}\f_{\bar\lambda}+M_\lambda\psi_\lambda=P_{\sN_{\bar\lambda}}h,\;
f_S=2i\IM \lambda(S-\lambda I)^{-1}P_{\sM_\lambda}h,\; h\in\sL.
\]
Thus,
$$\ran (I-Y_\lambda)\cap\ran D_{Y_\lambda ^*} =\{0\}\Longleftrightarrow
\ran  D_{Y_\lambda ^*}\cap (P_{\sN_{\bar\lambda}}-Y_\lambda P_{\sN_\lambda})\sL=\{0\}.$$

If $||M_\lambda||<1,$ then $\ran  D_{M_\lambda ^*}=\sN_{\bar\lambda}$.
Hence
\[
 D_{M_\lambda ^*}\f_{\bar\lambda}=P_{\sN_{\bar\lambda}}h-M_\lambda P_{\sN_\lambda}h,\;
\f_{\bar\lambda}= D_{M_\lambda ^*}^{-1}(P_{\sN_{\bar\lambda}}h-M_\lambda P_{\sN_\lambda}h)
\]



(2) Consider the corresponding rigged Hilbert space $\cH_T^+\subset\cH\subset\cH_T^-$.
By Lemma \ref{ygkcvby} there exists a Hilbert space $\cE$ and $\Gamma\in\bB(\cH_T^+,\cE)$ such that 
$
\cfrac{1}{2i}(T-T^\times)=\Gamma^\times \Gamma,
$
where $\Gamma^\times\in\bB(\cE,\cH_T^-)$ is the adjoint to $\Gamma$.

By Proposition \ref{zghbl} $T\in \bD_{\rm{sing}}$ $\Longrightarrow$ 
$\ran \Gamma^\times\cap\cH=\{0\}$. Then Lemma \ref{ygkcvby} yields the equality $\sS_T=\dom T\cap\dom T^*.$
\end{proof}

\begin{corollary}\label{btcrjyx}
Let $T$ be a maximal dissipative operator in the Hilbert space $\cH$. If  $T\in\bD_{{\rm{sing}}}$ and $\sS_T$ is non-trivial and non-dense, then the deficiency indices of the Hermitian part of $T$ are infinite.
\end{corollary}
\section{The Shtraus extensions of symmetric operators and their matrix representations} \label{nov3c} 
Let $S$ be a closed symmetric operator in the Hilbert space $\sH$, let
 $z\in \wh\rho(S)$ and let $\wt S_z$ be the Shtraus extension of $S$, see \eqref{slamb}.
Since
\[
\left\{\begin{array}{l}
\IM (\wt S_z(f_S+\f_z), f_S+\f_z)=\IM z ||\f_z||^2,\\[2mm]
\wt S_z(f_S+\f_z)-\bar z(f_S+\f_z)=(S-\bar z I)f_S+2i\IM z\, \f_z,\; f_S\in\dom S,\; \f_z\in\sN_z,
\end{array}\right.
\]
the operator $\wt S_z$ is maximal dissipative for $ z\in\dC_+$, maximal accumulative for $z\in\dC_-$, and selfadjoint for $z\in\wh\rho(S)\cap\dR$.
By Theorem \ref{dljgec} for the operator $M_z:\sN_z\to\sN_{\bar z}$ one has $M_z=0$. Hence $\wt S_z^*=\wt S_{\bar z}$ and from \eqref{rhfc11}
\[
\dom\wt S_z\cap\dom\wt S_z^*=\dom S\dot+P_{\sN_z}(\dom S)^\perp=\dom S\dot+P_{\sN_{\bar z}}(\dom S)^\perp,\; \sS_{\wt S_z}=\dom S,\; z\in\dC_+.
\]

Let $\wt Y_z=(\wt S_z-z I)(\wt S_z-\bar z I)^{-1}$ be the Cayley transform of $\wt S_z$. Then $\wt Y_z$ is the partial isometry of the form
$ 
\wt Y_z=(S-z I)(S-\bar z I)^{-1}P_{\sM_{\bar z}},$ $z\in\dC_+$.

Because $\sN_z$ is the eigen-subspace of $\wt S_z$, the subspace $\sM_{\bar z}$ is invariant for the resolvent of the adjoint operator $\wt S_{\bar z}$
and
\[
\dom \wt S_{\bar z}\cap\sM_{\bar z}=(\wt S_{\bar z}-z I)^{-1}\sM_{\bar z}.
\]
\begin{proposition}\label{extrrr}
Let $\wt S_z$ ($\IM z\ne 0$) be the  Shtraus extension of $S$ of the form \eqref{slamb}. Then
\begin{enumerate}
\item if $S$ is densely defined, then $\wt S_z\in\bD_{{\rm{sing}}}$ for each $z\in\dC_+$;
\item if $S$ is non-densely defined $S$, then $\wt S_z\notin \bD_{{\rm{sing}}}$ for each $z\in\dC_+.$

\end{enumerate}
\end{proposition}
\begin{proof}(1) If $S$ is densely defined, then by Theorem \ref{xfcnyck} each maximal dissipative extension of $S$ belongs to the class $\bD_{{\rm{sing}}}$.

(2) Assume $S$ is non-densely defined. Let $z\in\dC_+$. Because $M_z=0:\sN_z\to\sN_{\bar z}$ corresponds to the Shtraus extension $\wt S_z$ via the representation \eqref{dljgec}, we obtain $D_{M_z}=I_{\sN_z}$ and
\[
\ran D_{M_z}\cap(P_{\sN_z}-M^*_zP_{\sN_{\bar z}})(\dom S)^\perp=P_{\sN_z}(\dom S)^\perp=\sL_z,
\]
where the linear manifold $\sL_z$ is defined in \eqref{gjlghjcn}.
Hence, Theorem \ref{xfcnyck} yields $\wt S_z\notin \bD_{{\rm{sing}}}$.
\end{proof}

Further we give the operator-matrix representation of the Shtraus extension $\wt S_z$, $z\in\dC_+$. For this purpose we define linear operators:
\begin{equation}\label{al}
A_z:=P_{\sM_{\bar z}}\wt S_z\uphar\left(\sM_{\bar z}\cap\dom \wt S_z\right)
\end{equation}
and
\begin{equation}\label{bl}
B_z:=P_{\sN_{z}}\wt S_z\uphar\left(\sM_{\bar z}\cap\dom \wt S_z\right),\; \Gamma_z:=\cfrac{1}{2i\sqrt{\IM z}}\,B_z.
\end{equation}
Because $\wt S_z\uphar\sN_z=zI_{\sN_z}$, the domain $\dom\wt S_z$ admits the orthogonal decomposition
\[
\dom\wt S_z=(\sM_{\bar z}\cap\dom\wt S_z)\oplus \sN_z. 
\]
Hence, from \eqref{al} and \eqref{bl} we get
\begin{equation}\label{jrnj26b}
\left\{\begin{array}{l}
\dom \wt S_z=\dom A_z\oplus\sN_z,\\[2mm]
\wt S_z\begin{bmatrix}f\cr \f_z\end{bmatrix}=\begin{bmatrix} A_z&0\cr B_z&zI_{\sN_z}\end{bmatrix}\begin{bmatrix}f\cr\f_z\end{bmatrix}=\begin{bmatrix} A_z&0\cr 2i\sqrt{\IM z}\,\Gamma_z&zI_{\sN_z}\end{bmatrix}\begin{bmatrix}f\cr\f_z\end{bmatrix}=\begin{bmatrix}A_z f\cr 2i\sqrt{\IM z}\,\Gamma_zf+\f_z\end{bmatrix}\\[2mm]
\qquad f\in\dom A_z,\;\f_z\in\sN_z. 
\end{array}\right..
\end{equation}
Recall that a linear operator $A$ in the Hilbert space is called \textit{accretive} (see \cite{Ka}) if $\RE (Af,f)\ge 0$ for all $f\in\dom A$ and
$A$ is maximal accretive if it has no accretive extensions without exit. So, $A$ is a maximal accretive if and only if $T=iA$ is maximal dissipative
A linear operator $A$ is called \emph{sectorial with vertex at the origin $z=0$ and the
semi-angle $\alpha \in (0,\pi/2)$}, if its numerical range $W(A):=\{(Af,f), f\in\dom A, ||f||=1\}$ is contained in a
closed sector with semi-angle~$\alpha$, i.e.,
\[
  W(A) \subseteq 
  \left\{z\in\dC:|\arg z|\le \alpha\right\}\Longleftrightarrow |\IM (A u,u)|\!\le\! \tan\alpha \,\RE(A u,u)\;\forall u\!\in\!\dom A.
\]
Clearly, a sectorial operator is accretive; it is called maximal sectorial if
it is maximal accretive.

Note that if $A$ is a maximal sectorial operator, then the operator $T=iA$ is maximal dissipative and according to the second representation theorem \cite[Sect.~VI.2.6]{Ka} the quadratic form $\gamma_T$ admits the representation
\[
\gamma_T[u]=\IM (Tu,u)=\RE (Au,u)=||\Gamma u||^2\;\; u\in\dom T
\]
with a closable operator $\Gamma$.

\begin{theorem}\label{al2}
(1a) In the subspace $\sM_{\bar z}$ the operator $A_z$ ($z\in\dC_+)$ is maximal dissipative.
Moreover
\begin{equation}\label{al3}
\left\{\begin{array}{l}\dom A_z=\sM_{\bar z}\cap\dom \wt S_z=P_{\sM_{\bar z}}\dom S,\\[2mm]
 A_z(P_{\sM_{\bar z}}f_S)=P_{\sM_{\bar z}}Sf_S=(S-\bar z I)f_S+\bar z P_{\sM_{\bar z}}f_S,\\[2mm]
 \left(A_z(P_{\sM_{\bar z}}f_S\right), P_{\sM_{\bar z}}f_S)=\left(Sf_S,f_S\right)-\bar z||P_{\sN_z}f_S||^2,\;f_S\in\dom S.
 \end{array}
 \right.
\end{equation}
(1b) If the symmetric operator $S$ is nonnegative  (i.e., $(Sf,f)\ge 0$ for all $f\in\dom S$), then for $\RE z=0$ the operator $A_z$ is accretive and for $z: \RE z<0$ the operator $A_z$ is sectorial with the vertex at the origin and the semi-angle $\alpha=|\arg(-z)|.$

(2) The operator $B_z$ given by \eqref{bl} takes the form
\begin{equation}
\label{bl2}
\left\{\begin{array}{l}
\dom B_z=P_{\sM_{\bar z}}\dom S,\\
B_z(P_{\sM_{\bar z}}f_S)=-2i\IM z P_{\sN_z} f_S,\; f_S\in\dom S
\end{array}\right..
\end{equation}
and
\begin{equation} \label{bzaz}
||B_zf||^2=4\IM z\,\IM(A_zf,f) 
\;\;\forall f\in\sM_{\bar z}\cap \dom \wt S_z.
\end{equation}
(3) Let the operator $\Gamma_z$ be defined in \eqref{bl}, then
\begin{equation}\label{gamz}
\Gamma_z(P_{\sM_{\bar z}}f_S)=-\sqrt{\IM z}\,P_{\sN_z}f_S,\;\;f_S\in\dom S,
\end{equation}
the Hermitian domain and the Hermitian part of $A_z$ are given by
\begin{equation}\label{jrn27a}
\left\{\begin{array}{l}\sS_{A_z}=\ker\Gamma_z=\sM_{\bar z}\cap\dom S=(S-\bar zI)\dom S^2,\\
A_z\uphar\sS_{A_z}=P_{\sM_{\bar z}}S\uphar(\sM_{\bar z}\cap\dom S),
\end{array}\right.
\end{equation}
and
 $\{\sN_z,\Gamma_z\}$ is the boundary pair of $A_z.$

 (4) The adjoint $A_z^*$ is of the form
\begin{equation}\label{aladj}
\left\{\begin{array}{l}
\dom A_z^*=\left(\wt S_{\bar z}-z I\right)^{-1}\sM_{\bar z}=\dom \wt S_{\bar z}\cap\sM_{\bar z}=(S-\bar z I)\dom (\wt S_z S),\\[3mm]
A_z^*=\wt S_{\bar z}\uphar\dom A_z^*.
\end{array}\right.
\end{equation}

(5) The operator $A_z$ belongs to the class $\bD_{\rm sing}$ if and only if $\dom S$ is dense in $\sH.$

(6) The operator $S$ takes the form
\begin{equation}\label{jrn26}
\left\{\begin{array}{l}
\dom S=\left(I_{\sM_{\bar z}}+\cfrac{i}{2\IM z}B_z\right)\dom A_z=\left\{\begin{bmatrix}f\cr \cfrac{-\Gamma_z f}{\sqrt{\IM z}}\end{bmatrix}: f\in\dom A_z\right\},\\[2mm]
S\begin{bmatrix}f\cr-\cfrac{\Gamma_z f}{\sqrt{\IM z}}\end{bmatrix}=\begin{bmatrix}A_z f\cr-\cfrac{\bar z\, \Gamma_z f}{\sqrt{\IM z}}\end{bmatrix},\; f\in\dom A_z.
\end{array}\right.
\end{equation}

\end{theorem}

\begin{proof}
Let $f_S\in\dom S$. Then
\[
P_{\sM_{\bar z}}f_S=f_S-P_{\sN_z}f_S\in\dom S\dot+\sN_z=\dom\wt S_z.
\]
By definition \eqref{al} of the operator $A_z$ we get that $P_{\sM_{\bar z}}f_S\in\dom A_z$ .

Suppose $f=f_S+\f_z\in\sM_{\bar z}$, where $f_S\in\dom S$, $\f_z\in\sN_z$. Then $P_{\sM_{\bar z}}f=f=P_{\sM_{\bar z}}f_S$, i.e., $f\in\dom A_z.$ Hence $\dom A_z=P_{\sM_{\bar z}}\dom S.$

If $\dom S$ is dense in $\sH$, then the linear manifold $P_{\sM_{\bar z}}\dom S$ is dense in $\sM_{\bar z}$ for all $z\in\wh\rho(S)$.
If $S$ is not dense, then since $\sL\cap\sM_{\bar z}=\{0\}$ \cite[Lemma 2]{krasno} for all $z\in\dC\setminus\dR$ (where $\sL=(\dom S)^\perp$, we get that $P_{\sM_{\bar z}}\dom S$ is dense in $\sM_{\bar z}$ as well.
Furthermore,
\[
A_z(P_{\sM_{\bar z}}f_S)=P_{\sM_{\bar z}}\wt S_z f_S=P_{\sM_{\bar z}}(S-\bar z I)f_S+ \bar z P_{\sM_{\bar z}}f_S 
=(S-\bar z I)f_S+\bar zP_{\sM_{\bar z}}f_S,
\]
\[
\begin{array}{l}
\left(A_z(P_{\sM_{\bar z}}f_S), P_{\sM_{\bar z}}f_S\right)=\left(P_{\sM_{\bar z}}Sf_S, P_{\sM_{\bar z}}f_S\right)\\[3mm]
=\left(Sf_S,f_S-P_{\sN_z}f_S\right)=\left(Sf_S,f_S\right)- (P_{\sN_z}Sf_S,f_S)=(Sf_S,f_S)-(P_{\sN_z}(S-\bar zI)+\bar zP_{\sN_z}f_S,f_S)\\
=(Sf_S,f_S)-\bar z||P_{\sN_z}f_S||^2.
\end{array}
\]
Consequently
\begin{equation}\label{vybvfz1}
\IM\left(A_z(P_{\sM_{\bar z}}f_S), P_{\sM_{\bar z}}f_S\right)=\IM z||P_{\sN_z}f_S||^2,\; f_S\in\dom S.
\end{equation}
\[
\begin{array}{l}
\ran (A_z-\bar z I_{\sM_{\bar z}})
=\left\{P_{\sM_{\bar z}}Sf_S-\bar z P_{\sM_{\bar z}}f_S: f_S\in\dom S\right\}\\[3mm]=\left\{P_{\sM_{\bar z}}(S-\bar z I)f_S: f_S\in\dom S\right\} 
=\left\{(S-\bar z I)f_S: f_S\in\dom S\right\}=\sM_{\bar z}
\end{array}
\]
Since $z\in\dC_+$, it follows that the operator $A_z$ is maximal dissipative in the Hilbert space $\sM_{\bar z}.$
Thus, \eqref{al3} is proved.

Suppose $(Sf,f)\ge 0$ for all $f\in\dom S$. Then for
 $\RE z=0$ the last equality in \eqref{al3} shows that the operator $A_z$ is accretive.
If $\RE z<0$, then from \eqref{vybvfz1} and \eqref{al3}
\[
\left|\IM (A_zP_{\sM_{\bar z}}f_S, P_{\sM_{\bar z}}f_S)\right|\le \tan |\arg(-z)| \RE (A_z P_{\sM_{\bar z}}f_S, P_{\sM_{\bar z}}f_S)\;\forall f_S\in\dom S.
\]
Hence, $A_z$ is sectorial with vertex at the origin and the semi-angle $|\arg(-z)|$. 

Let $h=P_{\sM_{\bar z}}f_S$, $f_S\in\dom S$, then $h\in\dom \wt S_z$ and from \eqref{slamb}
\[
\begin{array}{l}
P_{\sN_z}\wt S_z h=P_{\sN_z}\wt S_z(f_S-P_{\sN_z}f_S)=P_{\sN_z}Sf_S-zP_{\sN_z}f_S=P_{\sN_z}(S-zI)f_S\\
=P_{\sN_z}(S-\bar zI)f_S-(z-\bar z)P_{\sN_z}f_S=-2i\IM z P_{\sN_z}f_S.
\end{array}
\]
Thus, relations \eqref{bl2} hold. Hence from \eqref{vybvfz1} we get \eqref{bzaz}.

It follows from \eqref{vybvfz1} that $\sN_z,$ is a boundary space and the operator $\Gamma_z$, defined in \eqref{gamz}, is a boundary operator for $A_z$.
Expression \eqref{jrn26} for $S$ is a consequence of \eqref{jrnj26b}, \eqref{al3}, and \eqref{bl2}.

The equality \eqref{jrn27a} is a consequence of \eqref{al3} and \eqref{gamz}.

Observe that the matrix representation \eqref{jrnj26b} yields that
\begin{equation}\label{dilres}
(A_z-\lambda I)^{-1}=P_{\sM_{\bar z}}(\wt S_z-\lambda I)^{-1}\uphar\sM_{\bar z}\;\;\forall \lambda\in\rho(\wt S_z).
\end{equation}

The relations in \eqref{dilres} and \eqref{aladj} follow from the equalities $\wt S^*_z=\wt S_{\bar z}$, $(\wt S^*_z-\xi I)^{-1}\sM_{\bar z}=\dom \wt S^*_z\cap\sM_{\bar z}$, $\xi\in\rho(\wt  S^*_z)$ and from the fact that $-A_z^*$ is maximal dissipative in $\sM_{\bar z}$.

Suppose $\psi_z\in\dom \Gamma^*_z$, where $\psi_z\in\sN_z$. Then there is $h\in\sM_{\bar z}$ such that
\[
(\Gamma_z(P_{\sM_{\bar z}}f_S),\psi_z)=(P_{\sM_{\bar z}}f_S, h)\Longleftrightarrow (-\sqrt{\IM z}\,P_{\sN_z}f_S,\psi_z)=(P_{\sM_{\bar z}}f_S, h).
\]
Hence
$
h+\sqrt{\IM z}\,\psi_z\in(\dom S)^\perp.
$ 
Since $\sM_{\bar z}\perp\sN_z$, it follows that
$$\dom\Gamma_z^*=\{0\}\Longleftrightarrow (\dom S)^\perp=\{0\}.$$
Now Proposition \ref{zghbl} gives $A_z\in\bD_{\rm sing}$ $\Longleftrightarrow$ $(\dom S)^\perp=\{0\}.$
\end{proof}
\begin{remark}\label{defind}
The Hermitian part $A_z\uphar\sS_{A_z}$ defined in \eqref{jrn27a} is the compression of the operator $S$ on the subspace $\sM_{\bar z}$. It is established in \cite[Theorem 5.1]{ACD} that
if $S$ is a closed densely defined symmetric operator with deficiency indices $\left<n_+, n_-\right>$, then the compression of $S$ on the finite co-dimensional subspace is symmetric, closed, densely defined and has the same deficiency indices. Consequently, if $\dim\sN_z<\infty$ then the operator $A_z\uphar\sS_{A_z}$ is closed densely defined symmetric operator with the same deficiency indices as $S$.
\end{remark}


\section{Maximal dissipative operators from the class $\bD_{{\rm{sing}}}$ whose domains intersect domains of their adjoint only at zero}
\label{nov12a}
The theorem below provides abstract examples of maximal dissipative operators whose domains have trivial intersections with domains of their adjoints.
\begin{theorem} \cite[Proposition 5.2, Theorem 5.3, Proposition 5.7]{ArlCAOT2023}\label{jrn3a}
\begin{enumerate}
\item Assume that the bounded selfadjoint operators $L$ and $M$ in $\cH$ satisfy the conditions
\begin{equation}\label{jgthfnkv}
\ker L=\{0\},\;\ker M=\{0\},\;\ran M\cap\ran L=\{0\}.
\end{equation}
Define the operator
\begin{equation}\label{yjtjkl}
A:=(M-iL^2)^{-1}.
\end{equation}
Then $A\in\bD_{\rm sing}$, $0\in\rho(A)$, $\dom A\cap\dom A^*=\{0\},$  and $\{\cH,LA\}$ is a boundary pair for $A.$

\item Let $Q$ be a bounded selfadjoint operator in the Hilbert space~ $\cH$,  $\ker Q=\{0\}$ and $\ran Q\ne \cH$. Let $\sM$ be a proper subspace in $\cH$ such that
\begin{equation}\label{yektd11}
\ran Q\cap\sM= \ran Q \cap\sM^\perp=\{0\}.
\end{equation}
Then  the operator
\begin{equation}  \label{Adissip}
A_0:=\left(-\left(Q\left(QP_\sM Q\right)^\half+\left(QP_\sM Q\right)^\half Q\right)-iQP_{\sM^\perp}Q\right)^{-1}
\end{equation}
is unbounded maximal dissipative, has bounded inverse, belong to the class $\bD_{\rm sing}$, $\dom A_0\cap\dom A_0^*=\{0\},$ and $\{\sM^\perp,P_{\sM^\perp}QA_0\}$
is a boundary pair of $A_0.$
\end{enumerate}
\end{theorem}

The next statement will be used in Section \ref{nov3b}.
\begin{corollary}\label{jrnj5b}
In an infinite-dimensional separable complex Hilbert space there exist contractions $X$ such that
\begin{equation}\label{yeyjtc}
\begin{array}{l}
\ker D_X=\{0\},\;\; \ran (I-X)\cap \ran D_{X^*}=\{0\}(\Longleftrightarrow \ran (I-X^*)\cap\ran D_X=\{0\}),\\[2mm]
\ran (I-X)\cap\ran (I-X^*)=\{0\}.
\end{array}
\end{equation}
\end{corollary}
\begin{proof}
Since $\cH$ is infinite-dimensional separable complex Hilbert space, then
 by von Neumann's theorem \cite{Neumann1929} for each unbounded selfadjoint operator $F$ in $\cH$ there exists
an unbounded selfadjoint operator $G$ such that $\dom F\cap\dom G=\{0\}.$ It follows that for a bounded selfadjoint operator $M$ with unclosed dense range there exists a bounded selfadjoint operator $L$ with unclosed dense range such that $\ran M\cap\ran L=\{0\}$ (see also \cite{FW}, \cite{Arl_ZAg_IEOT_2015}).

Schmüdgen \cite[Theorem 5.1]{schmud} established that

\textit{if $\cR$ is an operator range (the domain of an unbounded selfadjoint operator or a dense linear manifold, which is the range of a bounded nonnegative selfadjoint operator \cite{FW}) in $\cH$,  then there is a subspace $\sM$  such that}
\[
  \sM\cap\cR=\{0\}\quad\mbox{and}\quad \sM^\perp\cap\cR=\{0\}.
\]
Hence, for any bounded selfadjoint operator $Q$ in $\cH$ with $\ker Q=\{0\}$ and $\ran Q\ne \cH$
there exists a subspace $\sM$ such that \eqref{yektd11} holds.

It follows that one can construct maximal dissipative operators $A$ and $A_0$ given by \eqref{yjtjkl} and \eqref{Adissip}, respectively. By Theorem \ref{jrn3a} both these operators belong to the class $\bD_{\rm sing}$ and the equalities $\dom A\cap\dom A^*=\{0\}$ and $\dom A_0\cap\dom A_0=\{0\}$ hold.

Let
$
X=(A-iI)(A+i I)^{-1}
$ 
be the Cayley transform of $A$. Then Proposition \ref{zghbl} yields the equalities in \eqref{yeyjtc}.
The same properties has the Cayley transform of the operator $A_0$.
\end{proof}

\section{Operators from the class $\bD_{{\rm{sing}}}$ with non-trivial non-dense Hermitian domains} \label{nov3b}
According to Theorem \ref{xfcnyck} each maximal dissipative extension of a closed densely defined symmetric operator belongs to the class $\bD_{{\rm{sing}}}$ and if the Hermitian part of a maximal dissipative operator $T$ is non-dense and has at least one finite deficiency index, then
$T\notin \bD_{{\rm{sing}}}$. In this Section we consider the case of infinite deficiency indices of a non-densely defined symmetric operator and prove that it has maximal dissipative extensions of the class $\bD_{{\rm{sing}}}$.
Moreover, we prove the existence of maximal dissipative operators of the class $\bD_{{\rm{sing}}}$ with a non-trivial and non-dense Hermitian domain and having certain prescribed properties.
Our aim is establishing the following theorem.

\begin{theorem}\label{ceotcn} Let $\cH$ be an infinite-dimensional separable complex Hilbert space and let
$S$ be a closed non-densely defined symmetric operator in $\cH$ with infinite-dimensional deficiency indices. Then there are exist maximal dissipative extensions $T$ of $S$ in $\cH$ such that $T\in\bD_{{\rm{sing}}}$ and $S$ is the Hermitian part of $T$. 
\end{theorem}

In the proof we will need properties of the \textit{shorted operator}.

\subsection{The
 shorted operator} \label{nov21g}

For every nonnegative selfadjoint bounded operator $B$ in the Hilbert space
$\cH$ and every subspace $\cK\subset \cH\;$ M.G.~Kre\u{\i}n \cite{Kr}
defined the operator $B_{\cK}$ by the relation
\[
B_{\cK}=\max\left\{\,Z\in \bB(H):\,
    0\le Z\le B, \, {\ran}(Z)\subseteq{\cK}\,\right\}.
\]
The equivalent definition is:
\begin{equation}
\label{Sh1}
 \left(B_{\cK}f, f\right)=\inf\limits_{\f\in \cK^\perp}\left\{\left(B(f + \varphi),f +
 \varphi\right)\right\},
\quad  f\in H.
\end{equation}
 The operator $B_{\cK}$ is called
the \textit{shorted operator} (see \cite{And, AT}).

It is well known (see e.g.\cite{AT,KrO}) that a bounded $B\ge 0$ can be represented by $2\times 2$ block operator matrix of the form
\begin{equation}\label{ajhvvfn}
B=\begin{bmatrix}B_{11}&B^\half_{11}Y^*B^\half_{22}\cr B^\half_{22}YB^\half_{11}&B_{22}
\end{bmatrix}:\begin{array}{l}\cK\\\oplus\\\cK^\perp \end{array}\to
\begin{array}{l}\cK\\\oplus\\\cK^\perp \end{array},
\end{equation}
where $B_{11}\in\bB(\cK)$, $B_{11}\ge 0$, $B_{22}\in\bB(\cK^\perp)$, $B_{22}\ge 0$, $Y\in\bB\left(\cran B_{11},\cran B_{22}\right)$ is a contraction.

From \eqref{Sh1} one can derive that the operators $B_\cK$ and $B_{\cK^\perp}$ are given by the block matrices
\begin{equation}\label{shormat1}
B_\cK=\begin{bmatrix}B^\half_{11}\left(I-Y^*Y\right)B^\half_{11}&0\cr
0&0\end{bmatrix},\;  B_{\cK^\perp}=\begin{bmatrix}0& 0
\cr 0& B^\half_{22}\left(I-YY^*\right)B^\half_{22}\end{bmatrix}.
\end{equation}
It is proved in \cite{Kr} that
\begin{equation}\label{rangeSh}
 {\ran}(B_{\cK})^{\half}={\ran}B^{\half}\cap{\cK}.
\end{equation}
Hence and from \eqref{shormat1}
\begin{equation}\label{nov21b}
\ran B^{\half}\cap \cK =\ran \left(B^\half_{11}D_Y^2B^\half_{11}\right)^\half=B^\half_{11}\ran D_Y,
\end{equation}
\[
 \ran B^{\half}
 \cap \cK \quad \mbox{is dense}\quad\mbox{in}\quad \cK \iff  \ran B^{\half}_{11}\cap \ker D_Y=\{0\},
\]
\[
  B_{\cK}=0 \iff  \ran B^{\half}\cap \cK=\{0\}\iff  Y\quad\mbox{is an isometry}.
\]

Suppose $\dim\cK=\dim\cK^\perp=\infty$. It was established in
\cite[Corollary 3.6]{Arl_ZAg_IEOT_2015} that a bounded nonnegative selfadjoint operator $B$ in $H$ with $\ker B=\{0\}$
satisfies  the condition
\[
\cK\cap \ran B^\half=\cK^\perp\cap\ran B^\half=\{0\}
\]
if and only if the operator $B$ with respect to decomposition $H=\cK\oplus\cK^\perp$ takes the form
\[
B=\begin{bmatrix} W^2&WU\cr U^*W&U^*U  \end{bmatrix}=\begin{bmatrix}W\cr U^*\end{bmatrix} \begin{bmatrix}W& U\end{bmatrix}:\begin{array}{l}
\cK\\\oplus\\ \cK^\perp\end{array}\to \begin{array}{l}
\cK\\\oplus\\ \cK^\perp\end{array},
\]
where $W$ is a bounded and selfadjoint nonnegative operator in $\cK$ with $\ker W=\{0\}$, $U$ is a bounded operator acting from $\cK^\perp$ into $\cK$, $\ker U=\{0\}$, $\ker U^*=\{0\}$ and
$\ran W\cap\ran U=\{0\}.$

Suppose $\ker B_{11}=\{0\}$ and $\ker B_{22}=0$.
Then  from
\[
B\begin{bmatrix}f\cr h\end{bmatrix}=\begin{bmatrix}B_{11}&B^\half_{11}Y^*B^\half_{22}\cr B^\half_{22}YB^\half_{11}&B_{22}
\end{bmatrix}\begin{bmatrix}f\cr h\end{bmatrix}
=\begin{bmatrix} B_{11}f+B^\half_{11}Y^*B^\half_{22}h\cr B^\half_{22}YB^\half_{11}f+B_{22}h\end{bmatrix},
\]
we get that the equality $B\begin{bmatrix}f\cr h\end{bmatrix}=\begin{bmatrix}0\cr 0\end{bmatrix}$ is equivalent to the system of the equalities
\[
\left\{\begin{array}{l}
B^\half_{11}f+Y^*B_{22}^\half h=0\\
YB^\half_{11}f+B_{22}^\half h=0
\end{array}\right. \Longleftrightarrow \left\{\begin{array}{l} B^\half_{11}f=-Y^*B_{22}^\half h\\
 (I-YY^*) B^\half_{22}h=0\end{array}\right. \Longleftrightarrow \left\{\begin{array}{l} B^\half_{22}h=-YB_{11}^\half f\\
 (I-Y^*Y) B^\half_{11}h=0\end{array}\right..
\]
Besides
\begin{equation}\label{nov21f}
\begin{array}{l}
B\begin{bmatrix}f\cr h\end{bmatrix}\in\cK\Longleftrightarrow B^\half_{22}h=-YB_{11}^\half f,\;\; 
B\begin{bmatrix}f\cr h\end{bmatrix}\in\cK^\perp\Longleftrightarrow  B^\half_{11}f=-Y^*B_{22}^\half h.
\end{array}
\end{equation}

\subsection{Proof of Theorem \ref{ceotcn}}
Set $\sL:=(\cdom S)^\perp$. Fix $\lambda\in \dC_+$. Due to \eqref{gjlghjcn}, the subspaces $\sN_\lambda$ and $\sN_{\bar\lambda}$ admit the decompositions
$$\sN_\lambda=\overline\sL_\lambda\oplus\sN'_\lambda,\; \sN_{\bar\lambda}=\overline\sL_{\bar\lambda}\oplus\\\sN'_{\bar\lambda},
$$
where $\sL_\mu=P_{\sN_\mu}\sL.$

It is sufficient to
consider two cases: (1) both the semi-deficiency subspaces of $S$ are infinite-dimensional, (2) the orthogonal complement to $\dom S$ is infinite-dimensional and semi-deficiency indices of $S$ are arbitrary.

\vskip 0.3 cm

 (1) $\dim\sN'_\lambda=\dim \sN'_{\bar\lambda}=\infty.$

If $\dim\sL<\infty$, then $\dim \sL_{\lambda}=\dim\sL_{\bar\lambda}=\dim\sL<\infty.$ So, in this case $\sL_{\lambda}$ and $\sL_{\bar\lambda}$ are subspaces.

If $\dim\sL=\infty$, then it is possible that $\sL_{\lambda}$ is a subspace or is unclosed linear manifold (see Theorem \ref{ivekmzy}). Let $\overline{\sL_{\lambda}}$ be the closure and let $ \overline{V_\lambda}$ be the continuation of the operator $V_\lambda$ (see \eqref{forbis}) on $\overline{\sL_{\lambda}}$. Then $ \overline{V_\lambda}$ isometrically maps
$\overline{\sL_{\lambda}}$ onto $\overline{\sL_{\bar\lambda}}$.

Because $\dim \overline{\sL_{\lambda}}\le  \dim\sN'_\lambda$, according to the arguments above, there is a \textit{nonnegative selfadjoint contraction}  $B$ in $\sN_{\lambda}$  such that
\[
\ker B=\{0\},\; \ran B^\half\cap \overline\sL_{\lambda}=\{0\},
\]
\[
B=\begin{bmatrix}B_{11}&B^\half_{11}Y^*B^\half_{22}\cr B^\half_{22}YB^\half_{11}&B_{22}\end{bmatrix}:\begin{array}{l}\overline{\sL_{\lambda}}\\\oplus\\\sN'_\lambda \end{array}\to
\begin{array}{l}\overline{\sL_{\lambda}}\\\oplus\\\sN'_\lambda \end{array}.
\]
For this purpose it is necessary to choose nonnegative $B_{11}\in\bB( \overline\sL_{\lambda})$, $B_{22}\in\bB( \sN_\lambda')$ and an isometry $Y\in\bB(\overline\sL_{\lambda},\sN_\lambda')$ such that
\[
\ker B_{11}=\{0\},\;\ker B_{22}=\{0\},\; \ran B_{22}\ne \sN_\lambda',\; \ran Y\cap\ran B^\half_{22}=\{0\}.
\]
Let $U\in\bB(\sN'_\lambda,\sN'_{\bar\lambda})$ be an arbitrary unitary operator and let
\[
W:=U\oplus \alpha \overline{V_\lambda},
\]
where $|\alpha|=1$, $\alpha\ne 1$ and $V_\lambda P_{\sN_{\lambda}}f=P_{\sN_{\bar\lambda}}f,$ $f\in\sL$ (see \eqref{forbis}). Then $W$ is unitary mapping $\sN_\lambda$ onto $\sN_{\bar\lambda}$.
Set
\begin{equation}\label{xbcnjrjy}
M_\lambda:=W(I-B)^\half.
\end{equation}
Then
$$||M_\lambda g||^2=||(I-B)^\half g||^2=\left((I-B)g,g\right)=||g||^2-||B^\half g||^2<||g||^2\;\;\forall g\in\sN_\lambda\setminus\{0\}, $$
$$M^*_{\lambda}=(I-B)^\half W^{-1},\; D_{M_\lambda}=B^\half,\; \ran D_{M_\lambda}=\ran B^\half. $$
For $f\in\sL$ we have
\[
\begin{array}{l}
\left(P_{\sN_{\lambda}}-M^*_{\lambda}P_{\sN_{\bar\lambda}}\right)f=\left(I-(I-B)^\half W^{-1}V_\lambda\right)P_{\sN_{\lambda}}f\\
=\left(I-\bar\alpha(I-B)^\half V^{-1}_\lambda V_\lambda\right)P_{\sN_{\lambda}}f=\left(I-\bar\alpha(I-B)^\half\right) P_{\sN_{\lambda}}f.
\end{array}
\]
 For all $|\alpha|=1$, $\alpha\ne 1$ the operator $I-\bar\alpha(I-B)^\half$ has bounded inverse  $\left(I-\bar\alpha(I-B)^\half\right)^{-1}.$
If for some $h\in\sN_\lambda$ and some $f\in\sL$ holds the equality
\[
D_{M_\lambda}h=\left(P_{\sN_{\lambda}}-M^*_{\lambda}P_{\sN_{\bar\lambda}}\right)f,
\]
then
\[
B^\half h=\left(I-\bar\alpha(I-B)^\half\right)P_{\sN_{\lambda}}f=\left(I-\bar\alpha(I-B)^\half\right)g,
\]
where $g=P_{\sN_{\lambda}}f\in\overline\sL_\lambda$.
Hence
\[
\left(I-\bar\alpha(I-B)^\half\right)^{-1}B^\half h=B^\half \left(I-\bar\alpha(I-B)^\half\right)^{-1}h=g.
\]
Since $\ran B^\half\cap \overline\sL_{\bar\lambda}=\{0\}$, we obtain that $h=g=0.$
Thus,
\[
\ran  D_{M_\lambda }\cap (P_{\sN_{\lambda}}-M_\lambda ^* P_{\sN_{\bar\lambda}})\sL=\{0\}.
\]
Let $M_\lambda$ be given by \eqref{xbcnjrjy} and set
\[
Y_\lambda:=(S-\lambda I)(S-\bar \lambda I)^{-1}P_{\sM_{\bar\lambda}}+M_\lambda P_{\sN_\lambda},\; T:=(\lambda I-\bar\lambda Y_\lambda)(I-Y_\lambda)^{-1}.
\]
Then $T$ is $m$-dissipative extension of $S$ and Theorem \ref{xfcnyck} yields, that $S$ is the Hermitian part of $T$ and $T\in\bD_{{\rm{sing}}}$.

\vskip 0.3 cm

(2) $\dim\sL=\infty$. 

Now, since $\dim\sL=\infty$, also $\dim\sL_{\lambda}=\dim \sL_{\bar\lambda}=\infty.$ Fix $\lambda\in\dC_+$ and choose a contraction $X\in \bB(\overline{\sL_{\bar\lambda}})$ having
the properties \eqref{yeyjtc}. Then define $M_\lambda\in\bB(\sN_\lambda,\sN_{\bar\lambda})$ as follows:
\[
M_\lambda=\begin{bmatrix}X\overline {V_\lambda}&0\cr 0&0\end{bmatrix}:\begin{array}{l}\overline\sL_\lambda\\\oplus\\\sN'_\lambda\end{array}\to
\begin{array}{l}\overline\sL_{\bar\lambda}\\\oplus\\\sN'_{\bar\lambda}\end{array}.
\]
Hence
\[
\begin{array}{l}
D_{M^*_\lambda}= D_{X^*}P_{\overline\sL_{\bar\lambda}}\oplus P_{\sN'_{\bar\lambda}},\\[2mm]
(P_{\sN_{\bar\lambda}}-M_\lambda P_{\sN_\lambda})f=(V_\lambda -XV_\lambda)P_{\sN_\lambda}f=
 (I_{\overline{\sL_{\bar\lambda}}}-X)V_\lambda P_{\sN_\lambda}f,\; f\in\sL.
\end{array}
\]
Due to the choice of $X$ (see \eqref{yeyjtc}) we conclude, that $\ran D_{M^*_\lambda}\cap(P_{\sN_{\bar\lambda}}-M_\lambda P_{\sN_\lambda})\sL=\{0\}$.

Now Theorem \ref{xfcnyck} yields that the operator $T$ given by \eqref{ljghfci} with $M_\lambda$ defined above, belongs to the class $\bD_{{\rm{sing}}}$ and the Hermitian part of $T$ coincides with $S$.

The proof is complete.

\section{Lifting of a maximal dissipative operator and squares of their Hermitian parts} 
\label{constrone}
In this Section we recall constructions from \cite{ArlCAOT2023} of closed densely defined symmetric
operators having trivial domains of their squares and consider
 other cases of squares.

\subsection{The operators $\wt\cT_z,$ $z\in\dC_+$}
Let $T$ be an unbounded maximal dissipative operator in $\cH$ and let $\{\cE,\Gamma\}$ be a boundary pair associated with $T$. 

In the Hilbert space $\sH:=\cH\oplus\cE$ for each $z\in\dC_+$ define a linear operator $\wt \cT_z$ as follows: 
\begin{equation}\label{yjdjgth1}
\wt \cT_z=\begin{bmatrix} T&0\cr 2i\sqrt{\IM z}\,\Gamma& zI_\cE\end{bmatrix}:\begin{array}{l}\cH\\\oplus\\\cE\end{array}\to \begin{array}{l}\cH\\\oplus\\\cE\end{array}
,\;\;\dom \wt \cT_z=\dom T\oplus\cE.
\end{equation}
Clearly, the domain $\dom \wt\cT_z$ does not depend on $z\in\dC_+$ and the operator $\wt\cT_z$ for each $z\in\dC_+$ is the lifting (dilation) of the operator $T$, i.e.,
\[
\dom \wt\cT_z \cap\cH=\dom T,\; T=P_\cH\wt\cT_z\uphar(\cH\cap\dom\wt\cT_z).
\]
\begin{theorem}\label{djccnfy} \cite[Theorem 4.1]{ArlCAOT2023}.
 The following statements are valid:

(1) The operator $\wt \cT_z$
is maximal dissipative,
$\rho(\wt\cT_z)=\rho(T)\setminus\{z\}$
and 
\[
\begin{array}{l}
\IM (\wt \cT_z(f\oplus e),f\oplus e)=||\Gamma f+\sqrt{\IM z}\,e||^2_\cE\;\;\forall f\in\dom T,\; \forall e\in\cE,\\[2mm]
(T-\lambda I_\cH)^{-1}=P_\cH(\wt \cT_z-\lambda I_\sH )^{-1}\uphar\cH,\;\lambda\in\rho(T)\setminus\{z\},\\[2mm]
P_\cH\exp(it\wt \cT_z)\uphar\cH=\exp(it T), t\in\dR_+, \;\;\forall z\in\dC_+;
\end{array}
\]

(2) for each pair $z_1, z_2\in\dC_+$ holds the relation
\[
\wt\cT_{z_2}=\begin{bmatrix}I_\cH&0\cr 0&\sqrt{\cfrac{\IM z_2}{\IM z_1}}\,I_\cE\end{bmatrix}\left(\wt\cT_{z_1}+\begin{bmatrix}0&0\cr 0&\cfrac{\IM(\bar z_2 z_1)}{\IM z_2}\,I_\cE\end{bmatrix}\right)\begin{bmatrix}I_\cH&0\cr 0&\sqrt{\cfrac{\IM z_2}{\IM z_1}}\,I_\cE\end{bmatrix},
\]
in particular, the operators $\wt\cT_z,\wt\cT_{tz}$ are congruent for each $z\in\dC_+$ and each $t>0,$ $t\ne 1$.

\end{theorem}

One can verify directly from  \eqref{yjdjgth1} that the  adjoint operator $\wt \cT^*_z$ takes the form
\begin{equation}\label{cghz11}
\left\{\begin{array}{l}
\dom \wt\cT^*_z=\left\{\begin{bmatrix} f+2i\sqrt{\IM z}(T^\times-z I_\cH)^{-1}\Gamma^\times y\cr y\end{bmatrix}:\; f\in\dom T^*,\; y\in\cE\right\},\\[3mm]
\wt\cT^*_z\begin{bmatrix} f+2i\sqrt{\IM z}\,(T^\times-z I_\cH)^{-1}\Gamma^\times y\cr y\end{bmatrix}=\begin{bmatrix} T^*f+2iz\sqrt{\IM z}\,(T^\times-z I_\cH)^{-1}\Gamma^\times y\cr\bar z\, y\end{bmatrix}
\end{array}\right..
\end{equation}

\subsection{Symmetric operators $\cT_z$, $z\in\dC_+$}
In $\sH=\cH\oplus\cE$ for each $z\in\dC_+$ define the operator $\cT_z$ as follows
\begin{equation}\label{mods}
\left\{\begin{array}{l}
\dom \cT_z=\left\{\begin{bmatrix}f\cr-\cfrac{\Gamma f}{\sqrt{\IM z}}\,\end{bmatrix}:\; f\in\dom T\right\},\\[3mm]
\cT_z\begin{bmatrix}f\cr-\cfrac{\Gamma f}{\sqrt{\IM z}}\end{bmatrix}=\begin{bmatrix}Tf\cr-\cfrac{\bar z\, \Gamma f}{\sqrt{\IM z}}\end{bmatrix},\; f\in\dom T
\end{array}\right..
\end{equation}
From \eqref{yjdjgth1} and \eqref{mods} we get the equalities
\[
\dom T=\cH\cap\dom \wt\cT_z=P_\cH\dom \cT_z,\; T(P_\cH f_{\cT_z})=P_{\cH}\cT_z f_{\cT_z},\; f_{\cT_z}\in\dom\cT_z,
\]
\[
(\dom \cT_z)^\perp=\left\{\begin{bmatrix}\cfrac{\Gamma^*\f}{\sqrt{\IM z}}\cr\f \end{bmatrix}:\f\in\dom \Gamma^*\right\}.
\]
Equalities \eqref{kerna} and \eqref{vybvfz} yield that
\begin{equation}\label{jrnm7a}
\ran (\cT_z-\bar z I_\sH)=\cH,\; \dom\cT_z\cap \cH=\ker\Gamma=\sS_T.
\end{equation}
The definition \eqref{mods} of $\cT_z$ yields that the square $\cT^2_z$ is of the form
\begin{equation}\label{domsqu}
\dom \cT^2_z 
=\left\{\begin{bmatrix}(T-\bar z I_\cH)^{-1}g\cr-\cfrac{\Gamma (T-\bar z I_\cH)^{-1}g}{\sqrt{\IM z}}\end{bmatrix}:\;g\in
\sS_T\right\}.
\end{equation}
Hence $\dom \cT_z^2=\{0\}$ if and only if $\sS_T=\{0\}.$

\begin{theorem}\label{cnzcbv} cf. \cite[Theorem 4.2]{ArlCAOT2023}.

The following statements are valid:

(1) the operator $\cT_z$  defined in \eqref{mods} is symmetric and closed and
the deficiency subspace $\sN_\lambda(\cT_z)$ of $\cT_z$ takes the form
\[
\sN_\lambda(\cT_z)=\left\{\begin{bmatrix} h\cr y\end{bmatrix}: (T^\times-\lambda I_\cH)h=\frac{z-\lambda}{\sqrt{\IM z}}\Gamma^\times y, \; h\in\cH, y\in\cE\right\},
\]
in particular,
\begin{equation}\label{zltatrn}
\sN_z(\cT_z)=\cE;
\end{equation}

(2) $\dom \cT_z$ is the Hermitian domain  and $\cT_z$ is the Hermitian part of $\wt\cT_z$; 

(3) for distinct $z_1, z_2\in\dC_+$ holds the equality
\[
\cT_{z_2}=\begin{bmatrix}I_\cH&0\cr 0&\sqrt{\cfrac{\IM z_2}{\IM z_1}}\,I_{\cE}\end{bmatrix}\left(\cT_{z_1}+\begin{bmatrix}0&0\cr 0&\cfrac{\IM(\bar z_2 z_1)}{\IM z_2}\,I_{\cE}\end{bmatrix}\right)\begin{bmatrix}I_\cH&0\cr 0&\sqrt{\cfrac{\IM z_2}{\IM z_1}}\,I_{\cE}\end{bmatrix};
\]

(4) the operator $\cT_z$ is densely defined if and only if $T\in\bD_{\rm sing}$
and if this is case, then the adjoint operator $\cT^*_z$ is of the form
\begin{equation}\label{cjhzjg}
\left\{\begin{array}{l}
\dom \cT^*_z=\left\{\begin{bmatrix}f+2i\sqrt{\IM z}(T^\times-z I_\cH)^{-1}\Gamma^\times y\cr y+w\end{bmatrix}:\; f\in\dom T^*,\;y,w\in\cE\right\},\\[2mm]
\cT^*_z\begin{bmatrix}f+2i\sqrt{\IM z}(T^\times-z I_\cH)^{-1}\Gamma^\times y\cr y+w\end{bmatrix}=\begin{bmatrix}
T^*f+2iz\sqrt{\IM z}(T^\times-z I_\cH)^{-1}\Gamma^\times y\cr\bar z\, y+z w\end{bmatrix},
\end{array}\right.
\end{equation}
moreover, the domain $\dom \cT^*_z$ does not depend on $z\in\dC_+$; 

(5) the operator $\cT_z$ is simple if and only if the operator $T$ is simple;

(6)
assume $T\in\bD_{\rm sing}$, then
\begin{enumerate}
\item [\rm (I)]
the following are equivalent:
\begin{enumerate}
\def\labelenumi{\rm (\roman{enumi})}
\item [\rm (i)] $\dom \cT_z^2=\{0\}$,
\item [\rm (ii)] $\dom\cT_{z_1}\cap\dom\cT_{z_2}=\{0\}$ for some (then for any) $z_1, z_2\in\dC^+$ such that \\$\IM z_1\ne\IM z_2,$
\item [\rm (iii)]
$\dom T\cap\dom T^*=\{0\};$
\end{enumerate}
\item  [\rm (II)]
the following are equivalent:
\begin{enumerate}
\def\labelenumi{\rm (\roman{enumi})}
\item $\sS_T$ is dense in $\cH$,
\item $\dom\cT_z^2$ is a core of $\cT_z$.
\end{enumerate}
\end{enumerate}
\end{theorem}

\begin{proof}
Assertion (1)--(4) (except \eqref{cjhzjg}) and equivalences in (6)(I) are proved in  \cite[Theorem 4.2]{ArlCAOT2023}.

The description of $\cT_z^*$ in \eqref{cjhzjg} follow from \eqref{mods},
\eqref{cghz11} and the equalities
\[
\ker (\cT^*_z-z I)=\ker(\wt \cT_z-zI)=\cE,\;
\dom \cT^*_z=\dom\wt\cT^*_z\dot+\ker (\cT^*_z-z I).
\]

Let us show that $\dom\cT^*_z$ does not depend on $z$. Let two distinct complex numbers $z_1, z_2\in\dC_+$ be given.
Let $f_1\in\dom T^*$, $y_1, w_1\in\cE.$  Set
\[
\begin{array}{l}
y_2=\sqrt{\cfrac{\IM z_1}{\IM z_2}}\;\,y_1,\;w_2=y_1-y_2+w_1,\\
\f:= 2i\sqrt{\IM z_1}\,y_1=2i\sqrt{\IM z_2}\,y_2,\; \\
f_2=f_1+\left((T^\times-z_1 I_\cH)^{-1}\Gamma^\times \f-(T^\times-z_2 I_\cH)^{-1}\Gamma^\times \f\right)\\
\qquad\qquad=f_1-(z_1-z_2)(T^*-z_1 I_\cH)^{-1}(T^\times-z_2 I_\cH)^{-1}\Gamma^\times \f.
\end{array}
\]
Then $f_2\in\dom T^*$ and
\[
f_1+2i\sqrt{\IM z_1}(T^\times-z_1 I_\cH)^{-1}\Gamma^\times y_1=f_2+2i\sqrt{\IM z_2}(T^\times-z_2 I_\cH)^{-1}\Gamma^\times y_2.
\]
and $y_1+w_1=y_2+w_2$.  Now from \eqref{cjhzjg} it follows that $\dom\cT^*_{z_1}=\dom\cT^*_{z_2}.$

Suppose that the operator $\cT_z$ is not simple. Then $\cT_z=\cT_z^{(0)}\oplus\cT_z^{(1)}$, where $\cT_z^{(1)}$ is selfadjoint in a subspace $\cK^{(1)}\subset\dom\cT_z$.
From \eqref{mods}
 $\dom T=P_{\cH}\dom \cT_z$ and $(\cT_z-\bar z I)\dom \cT_z=\cH$. Hence
$\cK^{(1)}\subset\cH$ and $T\uphar\cK^{(1)}=\cT_z^{(1)}$. So $T$ is not simple

Now assume that $T$ is not simple,
\[
T=T^{(0)}\oplus T^{(1)},
\]
where $T^{(1)}$ is a selfadjoint operator in the subspace $\cH^{(1)}\subseteq\sS_T\subset\cH$.
 From \eqref{mods} it follows that $\cH^{(1)}\subset \dom \cT_z$ and $\cT_z\uphar\cH^{(1)}$ is a selfadjoint operator. Therefore $\cT_z$ is not simple.

The equivalences in (6)(II) follow from Lemma \ref{dbltnm} and equalities
 \eqref{jrnm7a}.

\end{proof}

\begin{remark}
\label{jrnj5a}
(1) Due to \eqref{zltatrn} the maximal dissipative operator $\wt\cT_z$ is the Shtraus extension (see \eqref{slamb}) of the symmetric operator $\cT_z$, i.e.,
\[
\dom\wt\cT_z=\dom \cT_z\dot+\cE,\; \wt\cT_z\uphar\dom\cT_z=\cT_z,\wt\cT_z\uphar\cE=zI_\cE.
\]
Since $\cH=\ran (\cT_z-\bar z I_\sH)$, the operator $T$ is the compression of $\wt\cT_z$ of the form \eqref{al}.

(2) From  \eqref{ctyn29a}, \eqref{orn11}, \eqref{ufvvfl}, \eqref{cghz11}, and \eqref{cjhzjg} we conclude that
\begin{equation}\label{jrn29b}
\cL_T=P_\cH\dom \wt\cT_z^*=P_{\cH}\dom \cT^*_z\;\;\forall z\in\dC_+.
\end{equation}
\end{remark}

\begin{proposition}\label{ahblajhv}
Let the operator $\cT_z$ be given by \eqref{mods}. Assume, in addition, that the operator $T$ is accretive, then $\cT_z$ is nonnegative for $\RE z\ge 0$ and for $z$ such that $\arg z\in (\pi/2,\pi-\alpha]$ in the case when
$T$, in addition, is $\alpha$-sectorial ($\alpha\in (0,\pi/2))$. 

\end{proposition}
\begin{proof}
If $T$ is accretive, then from \eqref{mods} and \eqref{ghtct} for $\RE z\ge 0$ we get
\[
\left(\cT_z\begin{bmatrix}f\cr-\cfrac{\Gamma f}{\sqrt{\IM z}}\end{bmatrix}, \begin{bmatrix}f\cr-\cfrac{\Gamma f}{\sqrt{\IM z}}\end{bmatrix}\right)=
(Tf,f)+\cfrac{\bar z}{\IM z}\,||\Gamma f||^2_\cE=\RE (Tf,f)+\cfrac{\RE z}{\IM z}\,||\Gamma f||^2_\cE\ge 0.
\]

Suppose that $T$ is a maximal sectorial with vertex at the origin and the semi-angle $\alpha$, i.e.,
\[
0\le\IM(Tf,f)=||\Gamma f||^2_\cE\le  \tan\alpha\,\RE (Tf,f)\;\;\forall f\in\dom T.
\]
If $\arg z\in (\pi/2,\pi-\alpha]$, then
\[
\RE z<0,\;
-\cfrac{\IM z}{\RE z}\ge \tan \alpha
\]
and from \eqref{mods} for any $f\in\dom T$ we have
\[
\begin{array}{l}
\left(\cT_z\begin{bmatrix}f\cr-\cfrac{\Gamma f}{\sqrt{\IM z}}\end{bmatrix}, \begin{bmatrix}f\cr-\cfrac{\Gamma f}{\sqrt{\IM z}}\end{bmatrix}\right)
=-\cfrac{\RE z}{\IM z}\left(-\cfrac{\IM z}{\RE z}\,\RE (Tf,f)-\IM (Tf,f)\right)\\
\ge-\cfrac{\RE z}{\IM z}\left(\tan\alpha\,\RE (Tf,f)-\IM (Tf,f)\right)\ge 0.
\end{array}
\]
Thus, the symmetric operator $\cT_z$ is nonnegative.

\end{proof}

Using \eqref{mods} and applying Theorem \ref{cnzcbv} to the maximal dissipative operators $A$ and $A_0$, describing in Theorem \ref{jrn3a}, we arrive at the next two assertions, see \cite{ArlCAOT2023}.
\begin{theorem} \label{totjl} \cite[Theorem 6.4]{ArlCAOT2023}.
Let $\cH$ be an infinite-dimensional Hilbert space. Assume that bounded selfadjoint operators $L$ and $M$ in $\cH$ satisfy conditions \eqref{jgthfnkv} and let the operator $A$ be given by \eqref{yjtjkl}.
Define for $z\in\dC_+$ the following linear operator in the Hilbert space $\bH:=\cH\oplus\cH$
\[
\left\{\begin{array}{l}
\dom\cA_z=\left\{\begin{bmatrix}f\cr -\cfrac{LAf}{\sqrt{\IM z}}  \end{bmatrix}: f\in\dom A\right\},\\[3mm]
\cA_z\begin{bmatrix}f\cr -\cfrac{LA f}{\sqrt{\IM z}}  \end{bmatrix}=\begin{bmatrix}Af\cr - \cfrac{\bar z LA f}{\sqrt{\IM z}}  \end{bmatrix},\;
f\in\dom A.
\end{array}\right.
\]
Then $\cA_z$ is densely defined and closed symmetric operator,
$0\in\wh\rho(\cA_z)$,
and
for distinct $z_1, z_2\in\dC_+$ holds the equality
\[
\cA_{z_2}=\begin{bmatrix}I_\cH&0\cr 0&\sqrt{\cfrac{\IM z_2}{\IM z_1}}\,I_{\cH}\end{bmatrix}\left(\cA_{z_1}+\begin{bmatrix}0&0\cr 0&\cfrac{\IM(\bar z_2 z_1)}{\IM z_2}\,I_{\cH}\end{bmatrix}\right)\begin{bmatrix}I_\cH&0\cr 0&\sqrt{\cfrac{\IM z_2}{\IM z_1}}\,I_{\cH}\end{bmatrix}.
\]
Moreover,
\begin{enumerate}
\item $\dom \cA_z^2=\{0\},$
\item $\dom\cA_{z_1}\cap\dom\cA_{z_2}=\{0\}$ for any $z_1, z_2\in\dC_+$ such that $\IM z_1\ne\IM z_2.$
\end{enumerate}
\end{theorem}
\begin{theorem} \label{abstexam} \cite[Theorem 6.5]{ArlCAOT2023}.
Let the operator $A_0$ in $\cH$ be defined by
\eqref{Adissip}, where $Q$ is a bounded selfadjoint operator in the Hilbert space $\cH$,  $\ker Q=\{0\}$, $\ran Q\ne \cH$ and $\sM$ be a proper subspace in $\sH$ such that
$\ran Q\cap\sM= \ran Q \cap\sM^\perp=\{0\}$ (see \eqref{yektd11}).

Define for $z\in\dC_+$ the following linear operator in the Hilbert space $\sH:=\cH\oplus\sM^\perp$
\[
\left\{\begin{array}{l}
\dom\cS_z=\left\{\begin{bmatrix}f\cr -\cfrac{P_{\sM^\perp} QA_0 f}{\sqrt{\IM z}}  \end{bmatrix}: f\in\dom A_0\right\},\\[3mm]
\cS_z\begin{bmatrix}f\cr -\cfrac{P_{\sM^\perp} QA_0 f}{\sqrt{\IM z}}  \end{bmatrix}=\begin{bmatrix}A_0f\cr -\cfrac{\bar z P_{\sM^\perp} QA_0 f}{\sqrt{\IM z}}  \end{bmatrix},\;
f\in\dom A_0
\end{array}\right..
\]
Then
\begin{enumerate}
\item the operator $\cS_z$ is densely defined and closed symmetric;
\item $0\in\wh\rho(\cS_z)$;
\item for distinct $z_1, z_2\in\dC_+$ holds the equality
\[
\cS_{z_2}=\begin{bmatrix}I_\cH&0\cr 0&\sqrt{\cfrac{\IM z_2}{\IM z_1}}\,I_{\sM^\perp}\end{bmatrix}\left(\cS_{z_1}+\begin{bmatrix}0&0\cr 0&\cfrac{\IM(\bar z_2 z_1)}{\IM z_2}\,I_{\sM^\perp}\end{bmatrix}\right)\begin{bmatrix}I_\cH&0\cr 0&\sqrt{\cfrac{\IM z_2}{\IM z_1}}\,I_{\sM^\perp}\end{bmatrix};
\]
\item $\dom \cS_z ^2=\{0\};$
\item $\dom\cS_{z_1}\cap\dom\cS_{z_2}=\{0\}$  for any $z_1, z_2\in\dC_+$ such that $\IM z_1\ne\IM z_2.$
\end{enumerate}
\end{theorem}
Now we consider intermediate cases related to the domain of $\cT_z^2$.
\begin{theorem}\label{ctyn29}
Assume that $T\in\bD_{\rm sing}$ in $\cH$ and $\dom T\cap\dom T^*$ ($=\sS_T$ in this case) is non-trivial and non-dense in $\cH$.
Let the symmetric operator $\cT_z$ ($z\in\dC_+$) be given by \eqref{mods}.
Then
\begin{enumerate}
\item the following are equivalent:
\begin{enumerate}
\def\labelenumi{\rm (\roman{enumi})}
\item $\dom\cT_z^2$ is dense in $\sH$,
\item $\cL_T\cap(\dom T\cap\dom T^*)^\perp=\{0\};$

\end{enumerate}
\item the following are equivalent:
\begin{enumerate}
\def\labelenumi{\rm (\roman{enumi})}
\item $\dom\cT_z^2$ is non-dense in $\sH$,
\item $\cL_T\cap(\dom T\cap\dom T^*)^\perp \ne\{0\}.$
\end{enumerate}
\end{enumerate}
\end{theorem}
\begin{proof}

Assume $\begin{bmatrix}h\cr \f\end{bmatrix}\in\left(\dom\cT_z^2\right)^\perp$. Then from \eqref{domsqu} we get
\[
\left((T-\bar z I_\cH)^{-1}g,h\right)-\cfrac{1}{\sqrt{\IM z}}\left(\Gamma(T-\bar z I_\cH)^{-1}g,\f\right)=0\;\;\forall g\in\sS_T(=\dom T\cap\dom T^*).
\]
Hence
\[
\left(g,(T^\times-z I_\cH)^{-1}\left(h-\cfrac{\Gamma^\times\f}{\sqrt\IM z} \right)\right)=0\;\;\forall g\in\sS_T.
\]
Therefore
\begin{equation}\label{25bdec}
\begin{bmatrix}h\cr \f\end{bmatrix}\in\left(\dom\cT_z^2\right)^\perp \Longleftrightarrow (T^\times-z I_\cH)^{-1}\left(h-\cfrac{\Gamma^\times\f}{\sqrt\IM z} \right)\in
\sS_T^\perp.
\end{equation}
Because $\ran\Gamma^\times\cap\cH=\{0\}$ and we get that
\[
\left(\dom\cT_z^2\right)^\perp=\{0\}\Longleftrightarrow\left(\dom T^*\dot+(T^\times-zI_\cH)^{-1}\ran \Gamma^\times\right)\cap(\sS_T)^\perp=\{0\}.
\]

Since $(T^\times-zI_\cH)^{-1}\ran \Gamma^\times=\ran D_{Y_z}$ and (see \eqref{ctyn29a}, \eqref{orn11}, Corollary \ref{octob9a})
$$\dom T^*\dot+\ran D_{Y_z}=\cL_T,$$
 where $Y_z$ is the Cayley transform of $T,$
  we get the equivalence
 \[
 \left(\dom\cT_z^2\right)^\perp=\{0\}\Longleftrightarrow \cL_T\cap(\sS_T)^\perp=\{0\}.
 \]
 Since $\sS_T=\dom T\cap\dom T^*$ (because $T\in\bD_{\rm sing}$) we arrive at the equivalence of (i) and (ii) in (1) and (2).
\end{proof}

\section{Maximal dissipative operators of the class $\bD_{\rm sing}$ of special types}\label{nov28d}

In this Section we construct abstract examples of maximal dissipative operators of the class $\bD_{\rm sing}$ having additional properties.
Recall that for maximal dissipative operator $T$ we use the notation $\cL_T:=\ran\left(\IM \left(T^*-i I\right)^{-1}\right)^\half$ (see \eqref{ctyn29a}) and $\sS_T:=\{f\in\dom T\cap\dom T^*: Tf=T^*f\}$ (see \eqref{kerna2}); if $T\in\bD_{\rm sing}$, then Theorem \ref{xfcnyck} says that $\sS_T=\dom T\cap\dom T^*$.

\subsection{}


\textbf{(1) Maximal dissipative operator $T\in\bD_{\rm sing}$ such that $\sS_T$ is non-trivial and non-dense and
$
\cL_T\cap(\sS_T)^\perp=\{0\}.
$
}

Let $S_0$ be an unbounded closed densely defined symmetric operator in $\cH_0$, let $\sL$ be a Hilbert space and let $L\in\bB(\cH_0, \sL)$.
Assume, that (see \eqref{zwete1})
$ \ker L^*=\{0\}$ and $\ran L^*\cap \dom S^*_0=\{0\}$.
Then by Proposition \ref{ghbv1} the operator (see \eqref{erst11})
\[
S:=S_0+LS_0,\;\dom S=\dom S_0.
\]
is closed symmetric in the Hilbert space $\cH:=\cH_0\oplus\sL$ and the equality \eqref{ers12} holds, i.e.,
$\dom S^*\cap\sL=\{0\},$
where $S^*:\cH\to\cH_0$ is the adjoint to the operator $S:\cH_0\to\cH.$

Assume, in addition, that deficiency indices of $S$ are infinite. It is possible if, for instance, deficiency indices of $S_0$ are infinite or $\dim \sL=\infty$.

By Theorem \ref{ceotcn} one can construct in this case a maximal dissipative extension $T$ of $S$ such that $T\in\bD_{\rm sing}$ and $\sS_T=\dom S$.
Since $\dom T,\dom T^*\subset\dom S^*$ and $(T^\times -z I_\cH)^{-1}\ran \Gamma^\times\subset\sN_z(S)\subset\dom S^*$ (see Remark \ref{vjtyjdp}), we get
\[
\left(\dom T^*\dot+(T^\times-zI_\cH)^{-1}\ran \Gamma^\times\right)\cap \sL=\{0\}.
\]
By Definition \eqref{ctyn29a} and Proposition \ref{yjdmt} we get that $\cL_T\cap \sL=\{0\}$.

\textbf{(2) Maximal dissipative operator $T\in\bD_{\rm sing}$ such that
 $\sS_T$ is non-trivial and non-dense and
$
\cL_T\cap(\sS_T)^\perp\ne\{0\}.
$
}

 Let $S_0$ be an unbounded selfadjoint operator in $\cH_0$, let $\sL$ be an infinite-dimensional Hilbert space and let
$L\in\bB(\cH_0, \sL)$.
Then the operator
$
S:=S_0+L, \;\dom S=\dom S_0
$ 
 is closed symmetric in the Hilbert space $\cH:=\cH_0\oplus\sL$ and
$
\dom S^*=\dom S_0\oplus\sL.$ 

Using a construction in item (2) of Theorem \ref{ceotcn}, we get a maximal dissipative operator $T$ in $\cH$ which belongs to the class $\bD_{\rm sing}$ with $\sS_T=S$.

Since $\dom S\subset\dom T\subset\dom S^*$ and $\dom S^*=\dom S_0\oplus\sL$, the decomposition
$$\dom T=\dom S_0\oplus(\dom T\cap\sL)$$
holds. Besides, $\dom T\subset\cL_T$ (see \eqref{ctyn29a}). Hence $\cL_T\cap\sL\ne \{0\}.$


\subsection{}

In this subsection we give abstract constructions of some \textit{maximal sectorial and dissipative} operators $T$ of the class $\bD_{\rm sing}$ that have the properties:
\begin{enumerate}
\item 
$\sS_T$ is dense;
\item 
$\sS_T$ is non-trivial, non-dense and $\cL_T\cap(\sS_T)^\perp=\{0\}$;
\item 
$\sS_T$ is non-trivial, non-dense and $\cL_T\cap(\sS_T)^\perp\ne\{0\}$;
\item $\sS_T=\{0\}$ ($\Longleftrightarrow$ $\dom T\cap\dom T^*=\{0\}$).
\end{enumerate}

We will construct maximal sectorial and dissipative operators in $\cH$ of the form
\begin{equation}\label{dec7}
T=B^{-\half}(I+iG)B^{-\half}=\left(B^\half(I+iG)^{-1}B^\half\right)^{-1}
\end{equation}
with bounded selfadjoint operators $B$ and $G$, in an infinite-dimensional complex separable Hilbert space $\cH$ and
$$B\ge 0,\; \ker B=\{0\},\; \ran B\ne \cH,\; G\ge 0.$$

These constructions will use  von Neumann's and Schmüdgen's results mentioned in the proof of Corollary \ref{jrnj5b}.

The operator $T$ is unbounded maximal sectorial and dissipative and
\eqref{dec7} gives the following descriptions of $\dom T$ and $\dom T^*$:
\begin{equation}\label{nov22e}
\left\{\begin{array}{l}
\dom T=\left\{u\in\dom B^{-\half}:(I+iG)B^{-\half} u\in\dom B^{-\half}\right\}\\
\;=B^{\half}(I+iG)^{-1}\ran B^{\half},\; Tu=B^{-\half}(I+iG)B^{-\half}u,\; u\in\dom T,\\
\dom T^*=\left\{v\in\dom B^{-\half}:(I-iG)B^{-\half} v\in\dom B^{-\half}\right\}\\
\;=B^{\half}(I-iG)^{-1}\ran B^{\half},\; T^*v=B^{-\half}(I-iG)B^{-\half}v,\; v\in\dom T^*.
\end{array}\right.
\end{equation}
By the first representation theorem \cite{Ka} the operator $T$ is associated with the closed sectorial form
\[
\tau[u,v]:=\left((I+iG)B^{-\half}u,B^{-\half}v\right), \;\;u,v\in\dom\tau=\dom B^{-\half}=\ran B^\half.
\]
Besides, $\dom T$ and $\dom T^*$ are cores of $\dom B^{-\half}.$
Hence
\[
\IM\tau[u]=(GB^{-\half}u, B^{-\half}u)=||G^\half B^{-\half}u||^2,\; u\in\dom B^{-\half}.
\]
For constructions in items (I)--(III) we involve selfadjoint bounded and non-negative operators $B$ in $2\times 2$ block operator matrix form \eqref{ajhvvfn}.

 Let $\cK\subseteq\cH$ be a subspace of $\cH$, $\dim\cK=\infty$ and let
\begin{equation}\label{nov21d}\begin{array}{l}
B_{11}\in\bB(\cK),\;
B_{11}\ge 0,\; \ker B_{11}=\{0\},\; \ran B_{11}\ne\cK,\\
B_{22}\in\bB(\cK^\perp),\;
B_{22}\ge 0,\; \ker B_{22}=\{0\}. 
\end{array}
\end{equation}
 Let $\alpha\in (0,\half\pi)$ and let
\begin{equation}\label{nov25a}
F\in \bB(\cK), \; F=F^*,\; \ker F=\{0\},\; ||F||\le \sqrt{\tan\alpha},\;G:=F^2P_{\cK}. 
\end{equation}
Hence,
\[
T=B^{-\half}(I+iF^2P_{\cK})B^{-\half}=\left(B^\half(I+iF^2P_{\cK})^{-1}B^\half\right)^{-1}.
\]
and due to $||G||=||F^2||\le\tan\alpha$ the operator $T$ is maximal sectorial with the semi-angle $\alpha,$
in addition, $T$ is dissipative,
\begin{equation}\label{nov25d}
\gamma_T[u]=\IM (Tu,u)=\IM\tau[u]=||FP_{\cK}B^{-\half}u||^2,\; u\in\dom T,
\end{equation}
and
\[
\sS_T=\ker\gamma_T=\left\{ u\in\dom T: B^{-\half} u\in\cK^\perp\right\}.
\]
Since $\dom T$ is a core of $\dom B^{-\half}$ and $\ker F=\{0\}$, the pair
$$\{\cK, FP_{\cK}B^{-\half}\}$$
is the boundary pair for $T$, i.e., $\cK$ is the boundary space and $\Gamma=FP_{\cK}B^{-\half}:\dom T\to\cK$ is the boundary operator.

We will construct $B$ such that
\begin{equation}\label{dek10a}
\ran B^\half\cap\cK^\perp\quad\mbox{is dense in}\quad\cK^\perp.
\end{equation}

Let us show that under the condition \eqref{dek10a} the equality
\begin{equation}\label{nov24b}
\sS_T=B^\half(\ran B^\half\cap\cK^\perp).
\end{equation}
holds.

 If $h=B^\half(B^\half g)$, where $\f:=B^\half g$ and $\f\in\cK^\perp$, then $h=B^{\half}\f$ and
\[
(I\pm iF^2 P_{\cK})B^{-\half}h=(I\pm iF^2 P_{\cK})\f=\f\in\cK^\perp\cap\ran B^{\half}.
\]
Hence and from \eqref{nov22e}
\[
h\in\dom T\cap\dom T^*, \; Th=T^*h=B^{-\half}\f=g.
\]
Therefore $h\in\sS_T.$

Conversely, if $h\in\sS_T$, then $h\in\dom T\cap\dom T^*$ and $Th=T^*h$.  It follows that
\[
\begin{array}{l}
\left\{\begin{array}{l} h\in\dom B^{-\half}\\
 (I\pm iF^2P_{\cK})B^{-\half} h\in\dom B^{-\half}\\
  B^{-\half}(I+iF^2P_{\cK})B^{-\half}h=B^{-\half}(I-iF^2P_{\cK})B^{-\half}h
 \end{array}\right.\Longleftrightarrow \left\{\begin{array}{l}h\in\dom B^{-1}\\F^2P_{\cK}B^{-\half}h=0\end{array}\right.\\[3mm]
 \qquad\Longleftrightarrow \left\{\begin{array}{l}B^{-\half}h\in\cK^\perp\\
 h\in\dom B^{-1}\end{array}\right.\Longleftrightarrow \left\{\begin{array}{l}
 h=Bg\\
 g\in\ran B^\half\cap\cK^\perp\end{array}\right.
\Longleftrightarrow h\in B^\half(\ran B^\half\cap\cK^\perp) \end{array}.
\]
Thus, \eqref{nov24b} holds. Then
\begin{equation}\label{nov22a}
(\sS_T)^\perp=B^{-\half}\left\{\ran B^\half\cap\cK\right\}.
\end{equation}
Clearly $B^{-\half}\left\{\ran B^\half\cap\cK\right\}\subseteq \sS_T^\perp$.
Let $f\in \sS_T^\perp$. Then for $h=B^\half\f$, where $\f\in \ran B^\half\cap\cK^\perp$ we have
\[
0=(f, B^\half \f)=(B^\half f,\f ).
\]
Since
$\ran B^\half\cap\cK^\perp$ is dense in $\cK^\perp$ we get that $B^\half f\in\cK.$
i.e.,
$B^\half\sS_T^\perp\subseteq\cK. $  
Hence \eqref{nov22a} is valid.

\vskip 0.3 cm

 \textbf{(1) The Hermitian domain $\sS_T$ is dense}.

 Assume
$$\dim\cK^\perp=\infty,\; \ran B_{22}\ne\cK^\perp$$ and
$$Y\in\bB(\cK,\cK^\perp)\quad\mbox{is isometry such that}\quad\ran Y\cap\ran B^\half_{22}=\{0\}.$$
It is possible to choose such $Y$ due to Schmüdgen's result \cite[Theorem 5.1]{schmud} (see the proof of Corollary \ref{jrnj5b}).

 Then \eqref{nov21f}, \eqref{nov21d}, and \eqref{nov21e}
yield that $B$ is nonnegative selfadjoint operator in $\cH$ and
\begin{equation}\label{nov22bc}
\ker B=\{0\},\;\ran B\ne \cH,\;
\ran B^\half\cap\cK= \{0\}.
\end{equation}
Because $I-YY^*=P_{(\ran Y)^\perp}$, from \eqref{shormat1} and $\ran Y\cap\ran B_{22}^\half=\{0\}$, we get  that \\$\ker(B_{22}^\half P_{(\ran Y)^\perp}B^\half_{22})=\{0\}$ and therefore \eqref{rangeSh} gives that \eqref{dek10a} is valid.

From \eqref{nov22bc} we get that $\sS_T^\perp=\{0\},$ i.e., $\sS_T$ is dense in $\cH$.
Then density of $\ker\gamma_T$ (as well as \eqref{nov22a} and \eqref{nov22bc}) implies $T\in\bD_{\rm sing}.$
\vskip 0.3 cm

\textbf{(2) The Hermitian domain $\sS_T$ is non-trivial, non-dense and $\cL_T\cap\sS_T^\perp=\{0\}$}.

Let $\dim\cK^\perp=\infty$. Choose $B_{22}\in\bB(\cK^\perp)$ and $Y\in\bB(\cK,\cK^\perp)$ such that 
\begin{equation}\label{nov21e}
\ran B^\half_{22}\ne \cK^\perp,\;||Yf||<||f||\;\;\forall f\in\cK\setminus\{0\},\;\ker Y=\{0\},\;\ran Y\cap\ran B_{22}^\half=\{0\}.
\end{equation}
In addition to \eqref{nov25a} assume that
\begin{equation}\label{nov25b}
\ran F\cap\ran B^\half_{11}=\{0\}.
\end{equation}

Then \eqref{ajhvvfn}, \eqref{nov21f}, \eqref{nov21d}, and \eqref{nov21e}
yield that $B$ is a nonnegative selfadjoint operator in $\cH=\cK\oplus\cK^\perp$ and
\begin{equation}\label{nov22b}
\ker B=\{0\},\;\ran B\ne \cH,\;
\ran B\cap\cK= \{0\}.
\end{equation}
Besides from \eqref{rangeSh} and \eqref{nov21b}
\[
\ran B^\half\cap\cK=B_{11}^\half\ran D_Y, \;\ran B^\half\cap\cK^\perp=B_{22}^\half\ran D_{Y^*}.
\]
Therefore
\[
\overline{\ran B^\half\cap\cK}=\cK,\;\overline{\ran B^\half\cap\cK^\perp}=\cK^\perp.
\]
Moreover, condition \eqref{nov25b} implies
\begin{equation}\label{nov25c}
\ran F\cap\ran B^{\half}=\{0\}.
\end{equation}
Since in this case $\ran B^\half\cap\cK\ne \{0\}$, the equality \eqref{nov22a} implies $\sS_T^\perp\ne\{0\}.$ 
Due to \eqref{nov25d} and \eqref{nov25c}
we get that $\dom \Gamma^*=\{0\}.$ By Proposition \ref{zghbl} the form $\gamma_T$ is singular, i.e. $T\in\bD_{\rm sing}.$

Let us show the inclusion
\begin{equation}\label{janu15}
\cL_T=\ran\left(\IM \left(T^*-i I\right)^{-1}\right)^\half\subseteq\ran B^{\half}.
\end{equation}
Actually from \eqref{nov22e}
\[\begin{array}{l}
\IM \left(T^*-i I\right)^{-1}=\IM \left(T^{*-1}(I-i T^{*-1})^{-1}\right)\\
=\cfrac{1}{2i}(I-i T^{*-1})^{-1}\left(T^{*-1}(I+iT^{-1})-(I-i T^{*-1}) T^{-1}\right)(I+i T^{-1})^{-1}\\
=(I-i T^{*-1})^{-1}\left(B^\half((I-iG)^{-1}G(I+iG)^{-1}B^\half+B^\half(I-iG)^{-1}B(I+iG)^{-1}B^\half\right)(I+i T^{-1})^{-1}.
\end{array}
\]
Taking into account the equality $(I+i T^{-1})^{-1}= (I+iB^\half (I+iG)^{-1}B^\half)^{-1}$, for all $f\in\cH$ we have
\[
\begin{array}{l}
\left(\left(\IM \left(T^*-i I\right)^{-1}\right)f,f \right)=||F(I+iG)^{-1}B^\half(I+iB^\half (I+iG)^{-1}B^\half)^{-1} f||^2\\
+||B^\half(I+iG)^{-1}B^\half (I+iB^\half (I+iG)^{-1}B^\half)^{-1} f||^2.
\end{array}
\]
Since
\[
(I+iG)^{-1}B^\half(I+iB^\half (I+iG)^{-1}B^\half)^{-1}=(I+ i(I+iG)^{-1}B)^{-1}(I+iG)^{-1} B^\half,
\]
we get
\[
\begin{array}{l}
\left(\left(\IM \left(T^*-i I\right)^{-1}\right)f,f \right)=||F(I+ i(I+iG)^{-1}B)^{-1}(I+iG)^{-1} B^\half f||^2\\
+ ||B^\half(I+ i(I+iG)^{-1}B)^{-1}(I+iG)^{-1} B^\half f||^2\le C ||B^\half f||^2 
\end{array}
\]
with some positive number $C$. Therefore, the Douglas lemma \cite{Doug} implies the inclusion \eqref{janu15}.

If $h\in\cL_T\cap(\sS_T)^\perp$, then there is $x\in\cH$ such that $h=B^\half x\in(\sS_T)^\perp$ and from \eqref{nov22a} $B^\half x\in  B^{-\half}\{\ran B^\half\cap\cK\}$.
Hence $Bx\in\ran B^\half\cap\cK\subset\cK$. But \eqref{nov22b} and $\ker B=\{0\}$ imply $x=0$.

Thus $\cL_T\cap(\sS_T)^\perp=\{0\}$.

\vskip 0.2cm

\textbf{(3) The Hermitian domain $\sS_T$ is non-trivial, non-dense and $\cL_T\cap(\sS_T)^\perp\ne\{0\}$}.

 \textit{Suppose that $\dim \cK^\perp<\infty$ and assume that conditions \eqref{nov25a}, \eqref{nov25b} and $||Y||<1$ are valid}.

It follows that
$$\ker B=\{0\},\;\ran B\ne \cH.$$
Moreover, expressions \eqref{nov24b} and \eqref{nov22a} remains valid. Now $\sS_T$ is finite dimensional and therefore the Hermitian part $S=T\uphar\sS_T=T^*\uphar\sS_T$ is a bounded closed (densely defined) symmetric operator. As before $T\in\bD_{\rm sing}$.
Hence the domain of the adjoint operator $S^*:\cH\to\sS_T$ coincides with $\cH$, i.e.
\[
\dom S^*=\dom \cH=\sS_T\oplus(\sS_T)^\perp.
\]
Because $T\supset S$ and $T^*\supset S$ we get that
\[
\dom T,\dom T^*\subset\dom S^*,
\]
and since
$$\sS_T\subset\dom T+\dom T^*\subset \ran\left(\IM \left(T^*-i I\right)^{-1}\right)^\half\subset\dom S^*,$$
we conclude that $\cL_T\cap\sS_T^\perp\ne\{0\}.$
\vskip 0.2cm
\textbf{(4) $\sS_T=\{0\}$}.

Let $\cK=\cH$, and let
$$\begin{array}{l}
B\in\bB(\cH), \;B\ge 0, \;\ker B=\{0\},\; \ran B\ne \cH,\\
F=F^*\in\bB(\cH), \;\ker F=\{0\},\; \ran F\cap\ran B^\half=\{0\},\; G=F^2.
\end{array}$$
Then 
 $T=B^{-\half}(I+iF^2P_{\cK})B^{-\half}\in \bD_{\rm sing}$, and \eqref{nov22e}, \eqref{nov24b} yield  $\sS_T=\dom T\cap\dom T^*=\{0\}$.

\vskip 0.2 cm

Note that the similar construction of unbounded maximal sectorial operators $T$ with $\dom T\cap\dom T^*=\{0\}$ can be found in \cite[Theorem 3.1]{ArlTret}.
\vskip 0.3cm
Summing up above constructions we arrive at the following theorem.
 \begin{theorem}\label{jrnzc2c}
 There exist maximal dissipative operators $T$ such that the symmetric operators $\cT_z$, $z\in\dC_+$ of the form \eqref{mods} are densely defined and
such that 
\begin{enumerate}
\item [\rm (I)] $\dom \cT_z^2$ is dense and is a core of $\cT_z$ for at least one (then for all) $z\in\dC_+$;
\item [\rm (II)] $\dom \cT_z^2$ is dense but is note a core of $\cT_z$ for at least one (then for all) $z\in\dC_+$;
\item [\rm (III)] $\dom\cT_z^2\ne \{0\}$ and is non-dense for at least one (then for all) $z\in\dC_+$;
\item [\rm (IV)] $\dom\cT_z^2= \{0\}$ for at least one (then for all) $z\in\dC_+$.
\end{enumerate}
Moreover, for each $\alpha\in (0,\half\pi)$ it is possible to choose 
$T$ 
such that the symmetric operator $\cT_z$ is nonnegative for $z$ with $\arg z\in (0,\pi-\alpha)$.
 \end{theorem}
\begin{proof}
The assertions follow from Theorem \ref{jrn3a}, Theorem \ref{cnzcbv}, Theorem \ref{totjl}, Theorem \ref{abstexam}, Theorem \ref{ctyn29}. The case of nonnegative $\cT_z$ follows from Proposition \ref{ahblajhv} and from the above constructions in this Section of maximal dissipative and sectorial operators.

\end{proof}
\section{The case of an arbitrary closed densely defined symmetric operator with infinite deficiency indices}
As has been mentioned in Introduction, if one of deficiency indices of a closed densely defined symmetric operator $S$ is finite, then necessary $\dom S^2$
is a core of $S$.  For a closed densely defined symmetric operator with infinite deficiency indices Lemma \ref{dbltnm} gives necessary and sufficient conditions in terms of the intersection $\sM_\lambda\cap \dom S$ ($\lambda$ is a point of regular type for $S$) of the cases $\dom S^2$ is a core of $S$ ($\Longleftrightarrow \sM_\lambda\cap \dom S$ is dense in $\sM_\lambda$)  and $\dom S^2=\{0\}$ ($\Longleftrightarrow \sM_\lambda\cap \dom S=\{0\}$).
Here we apply Theorem \eqref{ctyn29} for two intermediate cases related to $\dom S^2.$
\begin{theorem}\label{jrn29a}
Let $S$ be a closed densely defined symmetric operator with infinite deficiency indices in the Hilbert space. 

(1) The following are equivalent:
\begin{enumerate}
\def\labelenumi{\rm (\roman{enumi})}
\item $\dom S^2$ is dense in $\sH$ but not a core of $S$;
\item at least for one $\lambda\in\dC\setminus\dR$ (then for all such $\lambda$) the linear manifold $\sM_{\lambda}\cap\dom S$ is non-trivial and non dense in $\sM_{\lambda}$ and  holds the equality
\[
P_{\sM_{\lambda}}(\dom S^*)\cap \left(\sM_{\lambda}\ominus \left(\sM_{\lambda}\cap\dom S\right)\right)=\{0\}.
\]
\end{enumerate}
(2) The following are equivalent:
\begin{enumerate}
\def\labelenumi{\rm (\roman{enumi})}
\item $\dom S^2$ is non-trivial and non-dense in $\sH$;
\item at least for one $\lambda\in\dC\setminus \dR$ (then for all such $\lambda$) the linear manifold $\sM_{\lambda}\cap\dom S$ is non-trivial and non-dense in $\sM_{\lambda}$ and
\[
P_{\sM_{\lambda}}(\dom S^*)\cap \left(\sM_{\lambda}\ominus \left(\sM_{\lambda}\cap\dom S\right)\right)\ne\{0\}.
\]
\end{enumerate}
\end{theorem}
\begin{proof}

Fix $z\in\dC_+$. Due to \eqref{jrnj26b} and Theorem \ref{al2} the Shtraus extension $\wt S_z$ is the lifting of the form $\wt \cT_z$ (see \eqref{yjdjgth1}) of the maximal dissipative operator $T:=A_z$ defined in \eqref{al} in the Hilbert space $\cH:=\sM_{\bar z}$ and $\Gamma:=\Gamma_z$ (see \eqref{gamz}). Then the operator $S$ given by \eqref{jrn26} is of the form $\cT_z$ (see \eqref{mods}). Since $S$ is densely defined, due to Theorem \ref{al2} the operator $A_z$ belongs to the class $\bD_{\rm sing}$. Then we can apply Theorem \ref{ctyn29}.

Since $A_z\in\bD_{\rm sing}$, by Theorem \ref{xfcnyck} the equality $\sS_{A_z}=\dom A\cap\dom A^*_z$ holds and  \eqref{jrn27a} gives
$$\dom A\cap\dom A^*_z=\sM_{\bar z}\cap\dom S.$$
By \eqref{jrn29b}
\[
P_{\sM_{\bar z}}(\dom \wt S^*_z)=P_{\sM_{\bar z}}(\dom  S^*_z)=\cL_{A_z}.
\]
If $z\in\dC_-$ we may consider the operator $-S$ instead of $S$.

Thus, by Theorem \ref{ctyn29} assertions (i) and (ii) in (1) and (2) are equivalent.
\end{proof}

\begin{theorem}\label{yjzbr1a}
Let $S$ be a densely defined closed symmetric operator in the Hilbert space $\sH$.
Assume $n_+\ne 0$ and for $\mu, z\in\dC_+$ define a linear operator $S_{\mu,z}$ as follows
\begin{equation}\label{cl1}
\left\{\begin{array}{l}
\dom S_{\mu, z}= \left(P_{\sM_{\bar z}}+\sqrt{\cfrac{\IM z}{\IM \mu}}\,\,P_{\sN_z}\right)\dom S,\\
S_{\mu,z}\left(P_{\sM_{\bar z}}+\sqrt{\cfrac{\IM z}{\IM \mu}}\,P_{\sN_z}\right)f_{S}=\left(P_{\sM_{\bar z}}+\cfrac{\bar{\mu}}{\bar z}\sqrt{\cfrac{\IM z}{\IM \mu}}\,P_{\sN_z}\right)Sf_{S},\\
\qquad
f_{S}\in\dom {S}
\end{array}\right..
\end{equation}
Then the operator $S_{\mu, z}$ is symmetric, densely defined and closed and
\begin{enumerate}
\item at least for one pair $\{\mu, z\}$ (then for all) $\dom S_{\mu, z}^2 $ is a core of $\dom S_{\mu, z}$ if and only if 
$\dom S^2$ is a core of $\dom S$;
\item at least for one pair $\{\mu, z\}$ (then for all) $\dom S_{\mu, z}^2 $ is dense in $\sH$ but not a core of $\dom S_{\mu, z}$ if and only if 
$\dom S^2$ is  dense in $\sH$ but not a core of $\dom S$;
\item at least for one pair $\{\mu, z\}$ (then for all) $\dom S_{\mu, z}^2 $ is non-dense in $\sH$ if and only if 
$\dom S^2$ is non-dense in $\sH$;
\item  at least for one pair $\{\mu, z\}$ (then for all) $\dom S_{\mu, z}^2=\{0\} $ if and only if 
$\dom S^2=\{0\}$.
\end{enumerate}
\end{theorem}

\begin{proof}
Consider the operator
\[
\wt \cT_{\mu,z}=\begin{bmatrix} A_z&0\cr 2i\sqrt{\IM \mu}\,\Gamma_z& \mu I_{\sM_z}\end{bmatrix}:\begin{array}{l}\sM_{\bar z}\\\oplus\\\sN_z\end{array}\to \begin{array}{l}\sM_{\bar z}\\\oplus\\\sN_z\end{array}
,\;\;\dom \wt \cT_{\mu,z}=\dom A_z\oplus\sM_z,
\]
where $A_z$ and $\Gamma_z$ are given by \eqref{al} and \eqref{gamz}, respectively. Since $A_z$ is a maximal dissipative operator in $\sM_{\bar z}$ and
$\{\sN_z,\Gamma_z\}$ is its boundary pair (see Theorem \ref{al2}), the operator $\wt\cT_{\mu,z}$ is lifting of $A_z$ of the form \eqref{yjdjgth1} and by Theorem \ref{djccnfy}
it is maximal dissipative in the space $\sH=\sM_{\bar z}\oplus\sN_z$. According Theorem \ref{cnzcbv} the Hermitian part of $\wt\cT_{\mu,z}$ is the symmetric operator of the form \eqref{mods} and in our case is given by
\[
\begin{array}{l}
\dom \cT_{\mu,z}=\left\{\begin{bmatrix}f\cr-\cfrac{\Gamma_z f}{\sqrt{\IM \mu}}\,\end{bmatrix}:\; f\in\dom A_z\right\},\\[3mm]
\cT_{\mu,z}\begin{bmatrix}f\cr-\cfrac{\Gamma_z f}{\sqrt{\IM \mu}}\end{bmatrix}=\begin{bmatrix}A_zf\cr-\cfrac{\bar \mu\, \Gamma_z f}{\sqrt{\IM \mu}}\end{bmatrix},\; f\in\dom A_z.
\end{array}
\]
Note that $\cT_{z,z}=S.$
Since $\dom S$ is dense in $\sH$, Theorem \ref{al2} gives that $A_z\in\bD_{\rm sing}$ in the space $\sM_{\bar z}$ and then $\dom \cT_{\mu,z}$ is dense in $\sH$ and $\dom \cT_{\mu,z}^*$ does not depend on $\mu$ (see Theorem \ref{cnzcbv} (4)). Hence $\dom\cT_{\mu,z}^*=\dom S^*$. Besides
$$\sM_{\bar\mu}(\cT_{\mu,z})=\ran (\cT_{\mu,z}-\bar\mu I)=\sM_z.$$
Therefore (see \eqref{jrn29b})
\[
P_{\sM_{\bar z}}\dom \cT_{\mu,z}^*=P_{\sM_{\bar z}}\dom S^*=\dom A_z+\dom A^*_z.
\]
Besides
\[
\sM_{\bar\mu}(\cT_{\mu,z})\cap\dom \cT_{\mu,z}=\sM_{\bar z}\cap\dom S=\sS_{A_z}=\dom A_z\cap\dom A^*_z.
\]
Using expression \eqref{jrnj26b} for the operator $S$ and the expression \eqref{al3} for the operator $A_z$, we conclude that $\cT_{\mu,z}$ coincides with  the operator $S_{\mu,z}$ defined in \eqref{cl1}.

Now assertions of theorem follows from Theorem \ref{jrn29a}.
\end{proof}
\begin{remark}\label{nov01}
The operators $S_{\mu,z}$ have been introduced and studied in our paper \cite{ArlCAOT2023}. We called them "clones" of $S$.

Other clones of $S$ were defined and investigated in \cite{CAOT2021} for
$ 
z\in\dC\setminus\{\dR\cup i\dR\},\; \sN_z\ne\{0\}.
$ 
They take the form
\[
\left\{\begin{array}{l}
\dom S(z)=\left(\RE zP_{\sM_{\bar z}}+i\IM zP_{\sN_{z}}\right)\dom S=\left(\RE z I_\sH-\bar z P_{\sN_{z}}\right)\dom S ,\\[3mm]
S(z)\left(\RE z -\bar z P_{\sN_z} \right)f_S=\RE z Sf_S, \; f_S\in\dom S.
\end{array}\right.
\]
It is proved in \cite[Proposition 3.1, Theorem 3.2, Proposition 3.4]{CAOT2021} that
$$\dom S(z)^*=\dom S^*,\; \ran(S-\overline{\mu(z)}I)=\sM_{\bar z}, \;\dom S(z)\cap\sM_{\bar z}= (S-\bar zI)\dom S^2, $$
where $\mu(z)=iz\RE z/\IM z$.
It follows that conclusions (1)--(4) of Theorem \ref{yjzbr1a} remain valid for the operators $S(z)$ instead of $S_{\mu,z}$.
\end{remark}

\end{document}